\documentclass[11pt]{article}

    \usepackage[utf8]{inputenc}
    \usepackage[dutch,english]{babel}
    \usepackage{lipsum}
    \usepackage{mathtools}
    \usepackage{appendix}
    \usepackage{dsfont}
    \usepackage[colorinlistoftodos]{todonotes}
    \usepackage{graphicx}
	\usepackage{enumitem}
	\usepackage{mathrsfs}
    \usepackage{xfrac}
    \usepackage{amsmath, amsthm, amssymb}
    \usepackage{hyperref}
    \usepackage{import}
    \usepackage{upgreek}
    \usepackage{bigints}
    \usepackage{geometry}
    \usepackage{tikz}
    \usepackage{xcolor}
    
    \numberwithin{equation}{section}
    \DeclareMathOperator*\uplim{\overline{lim}}
    \DeclareMathOperator*\lowlim{\underline{lim}}

    \theoremstyle{definition}
    
    \newtheorem{Remark}{Remark}[section]
    
    \theoremstyle{plain}
    \newtheorem{Theorem}{Theorem}[section]
    \newtheorem{Proposition}[Theorem]{Proposition}
    \newtheorem{Lemma}[Theorem]{Lemma}
    \newtheorem{Corollary}[Theorem]{Corollary}

\newcommand{\E}{\mathbb{E}}
\newcommand{\p}{\mathbb{P}}

\newcommand{\dd}{\mathrm{d}}
\newcommand{\nn}{\nonumber}
\newcommand{\RTP}{\mathrm{RTP}}

\title{Hydrodynamic limit and large deviations for run-and-tumble particles with mean-field switching rates}
\author{Elena Pulvirenti, Frank Redig and Hidde van Wiechen}
\date{\today}

\begin{document}

\maketitle

\begin{center}
    \section*{Abstract}
    \end{center}
\begin{changemargin}{20mm}{20mm}
In this paper, we study run-and-tumble particles moving on two copies of the discrete torus (referred to as layers), where the switching rate between layers depends on a mean-field interaction among the particles. We derive the hydrodynamic limit of this model, as well as the large deviations from the hydrodynamic limit. Our main tool is the introduction of a weakly perturbed version of the system, whose hydrodynamic equations precisely characterize the trajectories associated with large deviations.
\end{changemargin}

\section{Introduction}
Run-and-tumble particles serve as simple models of active matter \cite{Tailleur}. Systems of active particles constitute an important class of non-equilibrium systems where at the microscopic scale energy is dissipated in order to produce directed motion. They show a rich phenomenology such as clustering and long-range order \cite{Tailleur}, \cite{vis}, \cite{schaller2010polar}.

In the mathematics literature on interacting particle systems, results on active particles are not very abundant (in contrast with the physics literature). To our knowledge, the first rigorous result on hydrodynamic limits for active particles is \cite{erignoux}, where a system of locally interacting active particles was studied.  In \cite{bodineau}, \cite{erignoux}, the authors prove the hydrodynamic limit (using the non-gradient method) and identify
a motility-induced phase separation and
a transition to collective motion from the equations obtained in the hydrodynamic limit. See also e.g.
\cite{noguchi2025spatiotemporal} for recent results in the physics literature on locally interacting active particles and collective effects therein.

In this paper we consider a simpler model where the interaction between the run-and-tumble particles is of mean-field type, i.e., via their empirical distribution. For this model we can then both prove the hydrodynamic limit and the large deviations from the hydrodynamic limit.
In our model, the particles move on the one-dimensional torus, and have an internal state which takes values $\pm 1$ and determines the direction of motion.
The interaction between the particles arises  implicitly via the flip rate at which particles flip their internal state which depends on the magnetization.

First, we start by deriving the hydrodynamic limit, which is a coupled system of partial differential equations for the densities of particles
with internal state $\pm 1$, and an ordinary differential equation for the ``magnetization''.
Second, we consider a weakly perturbed model where the influence of an external field which weakly depends on time and (microscopic) space is added. This field influences the rate at which particles flip their internal state and the rate at which they jump in the direction of their internal state.
This is reminiscent of the weakly asymmetric exclusion process \cite{KOV} which is an essential tool to study the large deviations for the trajectory of the density in the symmetric exclusion process (SEP). Contrary to the situation of the SEP, in our model, the perturbation does not act on the direction of the motion, which is always in the direction of the internal state.

We prove the hydrodynamic limit of this weakly perturbed model and use it to prove a large deviation principle for the trajectory of the densities in the original model. The technique of proof is based on a change of measure between the original and the weakly perturbed model and the associated  exponential martingale (the Radon-Nikodym derivative of the perturbed model w.r.t.\ the original model). Due to the mean-field character of the interaction, no super-exponential replacement lemmas are needed, i.e., the quantities appearing in the Radon-Nikodym derivative between the perturbed and unperturbed model are a function of the empirical densities and the magnetization.

The rest of our paper is organized as follows. In Section 2 we introduce the model, the weakly perturbed model and state the hydrodynamic limit of both. In Section 3 we prove large deviations for the trajectory of the densities. In Section \ref{proof hydro} we provide the proof of the hydrodynamic limits stated in Section 2.

\section{Run-and-tumble particles with mean-field switching rates}
In this section we describe the run-and-tumble particle model with mean-field switching rates. Later on we will also define a weakly perturbed version of this model which we will need for the large deviations. For both models, we will consider particles on the two-layered torus $V_N := \mathbb{T}_N \times S$, where $\mathbb{T}_N = \mathbb{Z}/N\mathbb{Z}$ is the discrete torus, and $S= \{-1,1\}$ which we will call the internal state space. We then say that a particle $(x,\sigma) \in V_N$ has position $x\in \mathbb{T}_N$ and internal state $\sigma \in S$. The parameter $N$ is a scaling parameter, and we will be interested in the limiting dynamics when $N\to\infty$.

We will consider processes of particle configurations $\{\eta^N_t:t\geq 0\}$ on the state space $\Omega_N = \mathbb{N}^{V_N}$. We denote by  $\eta^N_t(x,\sigma)$ the number of particles at site $(x,\sigma) \in V_N$ at time $t\geq 0$.  The Markovian dynamics is then described as follows:
\begin{itemize}
\item [i)] Active jump:
with rate $N$ a particle jumps from $(x,\sigma)$ to $(x+\sigma,\sigma)$.
\item[ii)] Internal state flip: a particle jumps from $(x,\sigma)$ to $(x,-\sigma)$, with a mean-field rate denoted by $c(\sigma, m_N(\eta_t^N))$
\end{itemize}
Here by ``mean-field'' we mean that the flip rate 
$c(\sigma, m_N(\eta_t^N))$ depends on the whole configuration $\eta$ only via its magnetization $m_N(\eta_t^N)$.
For $\eta\in\Omega_N$, this magnetization is defined via
 \begin{equation}\label{magn def}
        m_N(\eta) := \frac{1}{|\eta|} \sum_{x \in \mathbb{T}_N} \big(\eta(x,1) - \eta(x,-1)\big),
\end{equation}
where $|\eta|:= \sum_{(x,\sigma) \in V_N} \eta(x,\sigma)$ denotes the total number of particles in the configuration $\eta$. Here we also use the convention that if $|\eta|=0$ then $m_N(\eta)=0$.  We assume $c(\sigma, \cdot)$ to be bounded from above and below (away from zero) and Lipschitz-continuous in the second variable, i.e., there exists an $L>0$ such that for $\sigma \in \{-1,1\}$ and $m_1,m_2 \in [-1,1]$ we have that 
\begin{equation}
    |c(\sigma,m_1) - c(\sigma,m_2)| \leq L|m_1-m_2|.
\end{equation}

The generator of the process is described above,
working on functions $f:\Omega_N \to \mathbb{R}$, is defined as follows.
    \begin{align}\label{gen}
        \mathscr{L}_Nf(\eta) 
        &\nn= N\sum_{(x,\sigma)\in V_N} \eta(x,\sigma)\big[f(\eta^{(x,\sigma)\to (x+\sigma,\sigma)}) - f(\eta)\big]\\
        &\qquad \qquad  + \sum_{(x,\sigma)\in V_N} c(\sigma,m_N(\eta)) \eta(x,\sigma)\big[f(\eta^{(x,\sigma)\to(x,-\sigma)})-f(\eta)\big],
    \end{align}
where $\eta^{(x,\sigma) \to (y,\sigma')}$ denotes the configuration $\eta$ where a single particle has jumped from $(x,\sigma)$ to $(y,\sigma')$, if possible.

\begin{Remark}
    If we choose the rates $c(\sigma,m) \equiv 1$, then the particles do not interact with each other, and we recover the run-and-tumble particle process, studied, for instance, in \cite{Tailleur, hidde-ergodic, Hidde-Fluctuations, Ginkel, demaerel2018active}. An actual example of mean-field rates where particles do interact with one another is, for instance, given by the Curie-Weiss Glauber rates $c(\sigma,m) = e^{-\sigma \beta m}$ with $\beta>0$.
\end{Remark}

\subsection{Hydrodynamic limit}\label{hydro section}
For $N\in\mathbb{N}$, we define the empirical measure of a configuration $\eta\in \Omega_N$ by
    \begin{equation}\label{emp meas}
        \pi^N(\eta) := \frac{1}{N} \sum_{(x,\sigma)\in V_N} \eta(x,\sigma) \delta_{(\frac{x}{N},\sigma)},
    \end{equation}
where $\delta$ denotes the Dirac measure. For a given $\eta$, $\pi^N(\eta)$ is a positive Radon measure on the macroscopic space $V:=\mathbb{T}\times S$, where  $\mathbb{T} = [0,1]$ is the torus. We denote the space of positive Radon measures on $V$ by $\mathcal{M}_V$. 
Given $t\geq 0$, we further denote the empirical measure of the configuration $\eta^N_t$ as 
    \begin{equation}
        \pi^N_{t} := \pi^N(\eta_t^N) = \frac{1}{N} \sum_{(x,\sigma)\in V_N} \eta_t^N(x,\sigma)\delta_{(\tfrac{x}{N},\sigma)},
    \end{equation}
 For every $N\in\mathbb{N}$, this produces a process 
 $\{\pi^N_{t}:t\in [0,T]\}$ with trajectories in the Skorokhod space $D([0,T];\mathcal{M}_V)$. The corresponding space of test functions is given by,
    \begin{equation}
        C^\infty(V) := \{\phi :V\to \mathbb{R}\ \big|\ \phi(\cdot,\sigma) \in C^\infty(\mathbb{T}) \text{ for all $\sigma \in S$}\}.
    \end{equation}
For $\phi \in C^\infty(V)$ we now denote the pairing
    \begin{equation}
        \left< \pi^N_{t}, \phi\right> = \frac{1}{N} \sum_{(x,\sigma)\in V_N}\eta_t^N(x,\sigma) \phi(\tfrac{x}{N},\sigma).
    \end{equation}

For $N\in\mathbb{N}$ and  smooth $\varrho(x,\sigma):V\to\mathbb{R}_{\geq 0}$  we define the product Poisson measures 
    \begin{equation}\label{poissonini}
        \mu_N^{\varrho} = \bigotimes_{(x,\sigma) \in V_N} \mathrm{Pois}(\varrho\left(\tfrac{x}N,\sigma\right)), 
    \end{equation}
which is the local equilibrium measure for the non-interacting ($c(\sigma,m)\equiv1$) case  (see \cite{hidde-ergodic}). We assume that the process $\{\eta^N_t:t\geq0\}$ starts at $t=0$ from the configuration $\eta^N= \eta^N_0$ distributed according to $\mu_N^{\varrho}$, which fixes the initial density profile. More precisely, $\pi^N_{0}$ converges as $N\to\infty$ to the positive measure with density $\varrho(x,\sigma):V\to\mathbb{R}_{\geq 0}$, i.e., we have that for every $\phi \in C^\infty(V)$ and $\varepsilon>0$
    \begin{equation}
        \lim_{N\to\infty} \p_{N}^{\varrho}\left(\left|\langle \pi^N_{0},\phi\rangle - \sum_{\sigma \in S} \int \phi(x,\sigma)\varrho(x,\sigma)\dd x\right|>\varepsilon\right) = 0.
    \end{equation}
Here $\p_{N}^{\varrho}$ denotes the path-space measure of the process with initial distribution $\mu_N^{\varrho}$, where we omit the dependence on the terminal time $T$ since we assume that this is fixed. 

The question of the hydrodynamic limit is to find the limiting PDE for  $\pi^N_{t}$ as $N\to\infty$.  
We start with the following preliminary computation
    \begin{align}
        \mathscr{L}_N \langle\pi^N(\eta),\phi\rangle
        &\nn= N\sum_{(x,\sigma) \in V_N} \eta(x,\sigma) \left(\big\langle\pi^N\big(\eta^{(x,\sigma)\to(x+\sigma,\sigma)}\big), \phi\big\rangle - \langle \pi^N(\eta),\phi\rangle\right)\\
        &\ \ \ \ \ \ \ \ + \sum_{(x,\sigma) \in V_N} \eta(x,\sigma) c(\sigma, m_N(\eta))\left(\big\langle\pi^N\big(\eta^{(x,\sigma)\to(x,-\sigma)}\big), \phi\big\rangle - \langle \pi^N(\eta),\phi\rangle\right)\\
        &\nn=  \frac{1}{N} \sum_{(x,\sigma) \in V_N} \eta(x,\sigma) \left[N\big(\phi(\tfrac{x+\sigma}{N},\sigma) - \phi(\tfrac{x}{N},\sigma)\big) + c(\sigma,m_N(\eta))\big(\phi(x,-\sigma)- \phi(x,\sigma)\big) \right].
    \end{align}
From this computation, we observe that the evolution of the empirical measure  depends on the evolution of the magnetization $m_N(\eta^N_t)$. Note however that the magnetization, defined in \eqref{magn def}, can be expressed in terms of the empirical measure as follows, 
    \begin{equation}\label{mag from emp}
        m_N(\eta) = m(\pi^N(\eta)) := \frac{\langle\pi^N(\eta), \mathds{1}_{\sigma =1} - \mathds{1}_{\sigma=-1}\rangle}{\langle \pi^N(\eta),1\rangle}.
    \end{equation}
This representation of the magnetization motivates the following definition of the magnetization corresponding to a density $\varrho(x,\sigma): V\to \mathbb{R}_{\geq0}$ as 
    \begin{equation}\label{mag def}
        m(\varrho) := \frac{\langle \varrho, \mathds{1}_{\sigma=1} - \mathds{1}_{\sigma=-1}\rangle_{L^2(V)}}{\langle \varrho, 1\rangle_{L^2(V)}},
    \end{equation}
where $\langle\cdot,\cdot \rangle_{L^2(V)}$  denotes the inner product of $L^2$-functions on the space $V$, given by 
    \begin{equation}
        \langle \psi, \phi\rangle_{L^2(V)} = \sum_{\sigma \in S} \int_{\mathbb{T}} \psi(x,\sigma) \phi(x,\sigma)\dd x,
    \end{equation}
for any $\psi,\phi \in L^2(V)$. 
\begin{Theorem}\label{hydro theorem}
    $\p^\varrho_N(\pi^N_{\cdot}\in \cdot) \to \delta_{\alpha}$, where $\alpha\in D([0,T];\mathcal{M}_V)$ is the trajectory of measures with density $\varrho_t(x,\sigma)$ 
    which solves the following  equation,
    \begin{equation}\label{PDE}
            \dot{\varrho}_t(x,\sigma) = -\sigma \partial_x \varrho_t(x,\sigma) + c(-\sigma,m(\varrho_t))\varrho_t(x,-\sigma) - c(\sigma,m(\varrho_t))\varrho_t(x,\sigma),
    \end{equation}
    with initial condition $\varrho_0(x,\sigma) = \varrho(x,\sigma)$.
\end{Theorem}
This theorem will be a consequence of a more general hydrodynamic limit of a weakly perturbed modification of the model, which we will introduce in Section \ref{wam section}. Apart from the hydrodynamic limit of the empirical measure, we can also find a limiting equation of the magnetization.

\begin{Corollary}\label{mag hydro}
    $\p^\varrho_N(m(\pi^N_{\cdot})\in \cdot) \to \delta_{m(\varrho_\cdot)}$, where $\varrho_t$ solves \eqref{PDE}. Furthermore, $m_t := m(\varrho_t)$ solves the following equation, 
    \begin{equation}\label{mag-PDE}
        \dot{m}_t =  c(-1,m_t)\cdot  (1-m_t)-c(1,m_t) \cdot (1+m_t),
    \end{equation}
    with initial condition  $m_0 = m(\varrho)$.
\end{Corollary}
\begin{Remark}
    When considering the Curie-Weiss Glauber rates $c(\sigma,m) = e^{-\sigma \beta m}$ with inverse temperature $\beta>0$, the evolution of the magnetization is given by 
        \begin{equation}
            \dot{m}_t = 2\sinh(\beta m_t) -2m_t\cosh(\beta m_t).
        \end{equation}
    As $t\to\infty$ the process $m_t$ converges to a solution of the mean-field equation 
        \begin{equation}
            m^* = \tanh(\beta m^*). 
        \end{equation}
    For $\beta\leq 1$ the only solution is $m^*=0$. However, for $\beta>1$, there exist two nonzero solutions, and if $m_0\neq 0$, the process $m_t$ converges to either the  positive or negative solution, depending on the initial value $m_0$. This simulates a type of flocking behavior of the particles, where particles tend to move in the same direction. 
\end{Remark}

\subsection{Weakly perturbed model}\label{wam section}
We furthermore introduce a weakly perturbed version of our model, which will be a key tool for the study of large deviations since it produces the deviating trajectories which have a finite rate function. The weak perturbation will be parametrized by a time-dependent potential $H:[0,T]\times V \to \mathbb{R}$, which we will assume to be differentiable in time and continuous in space, and which we will denote by $H\in C^{1,0}([0,T]\times V)$. The time-dependent generator of this model, acting on functions $f:\Omega_N\to\mathbb{R}$, is given as follows:
\begin{align}\label{weak asym gen}
    \mathscr{L}^H_{N,t} f(\eta)  &= N \sum_{(x,\sigma)\in V_N} \eta(x,\sigma)e^{H_t(\frac{x+\sigma}{N},\sigma ) - H_t(\frac{x}{N},\sigma) }  \big[f(\eta^{(x,\sigma)\to (x+\sigma,\sigma)}) - f(\eta)\big]\\
    &\nn\ \ \ \ \ \ \ \ \ \ \ \ \ \ \ \ \ \ \ + \sum_{(x,\sigma)\in V_N}  \eta(x,\sigma)c(\sigma,m_N(\eta))e^{H_t(\frac{x}{N},-\sigma) - H_t(\frac{x}{N},\sigma)}\big[f(\eta^{(x,\sigma)\to(x,-\sigma)})-f(\eta)\big].
\end{align}
Note that for $H=0$ we recover the original model.
We will denote by $\p_N^{\varrho,H}$ the path-space measure of this process, where $\eta^N_0$ is distributed as $\mu^{\varrho}_N$ (cf.\ \eqref{poissonini}). We further abbreviate  $\widetilde{H}_t(x) := H_t(x,1) - H_t(x,-1)$. The hydrodynamic limit of this process is then given in the following theorem. 
\begin{Theorem}\label{weak hydro theorem}
    $\p^{\varrho,H}_N(\pi^N_{\cdot}\in \cdot) \to \delta_{\alpha^H}$, where $\alpha^H\in D([0,T];\mathcal{M}_V)$ is the trajectory of measures with density $\varrho^H_t(x,\sigma)$ 
    which solves the following  equation,
    \begin{equation}\label{weak-PDE}
            \dot{\varrho}^H_t(x,\sigma) = -\sigma \partial_x \varrho^H_t(x,\sigma) + c(-\sigma,m(\varrho^H_t))e^{\sigma \widetilde{H}_t(x)}\varrho^H_t(x,-\sigma) - c(\sigma,m(\varrho^H_t))e^{-\sigma \widetilde{H}_t(x)}\varrho^H_t(x,\sigma),
    \end{equation}
    with initial condition $\varrho^H_0(x,\sigma) = \varrho(x,\sigma)$.
\end{Theorem}
Note that Theorem \ref{hydro theorem}  follows from this Theorem \ref{weak hydro theorem} by choosing $H\equiv 0$. The proof of  Theorem \ref{weak hydro theorem} will be postponed to Section \ref{proof hydro}. Below we give the evolution of the magnetization under the perturbed dynamics. 
\begin{Corollary}\label{weak mag hydro}
    $\p^{\varrho,H}_N(m(\pi^N_{\cdot})\in \cdot) \to \delta_{m(\varrho^H_\cdot)}$, where $\varrho^H_t$ solves \eqref{weak-PDE}. Furthermore, $m^H_t := m(\varrho^H_t)$ solves the following equation, 
    \begin{equation}\label{weak mag-PDE}
        \dot{m}^H_t =  \frac{1}{\langle \varrho,1\rangle_{L^2(V)}} \left\langle\varrho^H_t, -2\sigma e^{-\sigma \widetilde{H}_t(x)}c(\sigma,m^H_t) \right\rangle_{L^2(V)}.
    \end{equation}
    with initial condition  $m_0 = m(\varrho)$.
\end{Corollary}
\begin{proof}
    Note that if $\alpha_N \to \alpha$  in $D([0,T]:\mathcal{M}_V)$ then $m(\alpha_N) \to m(\alpha)$ in $D([0,T]:[-1,1])$, hence  $m(\pi^N_\cdot) \xrightarrow{d} m_\cdot$ under $\p^{\varrho,H}_N$. Therefore, we can find an equation for the evolution of the magnetization from the evolution of $\varrho^H_t$. First note that $\langle\varrho_t^H,1\rangle_{L^2(V)} = \langle \varrho,1\rangle_{L^2(V)}$ for all $t\geq0$, which is due to the conservation of particles. Therefore, using the definition \eqref{mag def}, we find
        \begin{equation}\label{time mag}
            \dot{m}_t = \frac{\langle \dot{\varrho}^H_t, \mathds{1}_{\sigma=1} - \mathds{1}_{\sigma=-1}\rangle_{L^2(V)}}{\langle\varrho,1\rangle_{L^2(V){}}}
        \end{equation}
    By the periodic boundary conditions, we have that 
        \begin{equation}
            \langle -\sigma \partial_x\varrho_t^H, \mathds{1}_{\sigma=1}-\mathds{1}_{\sigma =-1}\rangle_{L^2(V)} = 0,
        \end{equation}
    hence \eqref{weak mag-PDE} follows by filling in \eqref{weak-PDE} into \eqref{time mag}.
\end{proof}

Corollary \ref{mag hydro} follows from Corollary \ref{weak mag hydro} by choosing $H\equiv0$ and using that for any density $\varrho$ we have that 
\begin{equation}
    \frac{\int_{\mathbb{T}}\varrho(x,\sigma) \dd x }{\langle \varrho,1\rangle_{L^2(V)}} = \tfrac{1}{2}(1+\sigma m(\varrho)).
\end{equation}
Note that, unlike in the unperturbed model, under the perturbed dynamics the evolution of the magnetization is not a closed equation, but depends on the density $\varrho^H_t$. This is because under the these dynamics the magnetization process is no longer a Markov process on its own, and additional information on the positions of the particles is required.

\section{Large deviations}
In this section we will prove a large deviation principle for the run-and-tumble particle process with mean-field switching rates. We start by defining the rate function $\mathcal{I}^{\varrho}: D([0,T];\mathcal{M}_V) \to [0,\infty]$, which is given in two parts
    \begin{align}\label{rate function}
        \mathcal{I}^{\varrho}(\widehat{\alpha}) = h_0^{\varrho}(\widehat{\alpha}_0) + \mathcal{I}_{tr}(\widehat{\alpha})
    \end{align}
Here $h_0^{\varrho}(\widehat{\alpha}_0)$ is the static part of the large deviation rate function, only depending on the measure at time $t=0$. It can be informally written as
    \begin{equation}
        h^\varrho_0(\widehat{\alpha}_0) = \mu^\varrho_N(\pi^N_{0} \approx \widehat{\alpha}_0),
    \end{equation}
i.e, it corresponds to the large deviation principle of the initial density profile $\pi^N_{0}$ under the starting distribution $\mu^\varrho_N$.
Since $\mu_N^{\varrho}$ is given by a product Poisson measure, the corresponding rate function is known, and given by 
    \begin{equation}\label{rateini}
        h_0^{\varrho}(\widehat{\alpha}_0) = \sup_\phi h_0^{\varrho}(\widehat{\alpha}_0;\phi), \ \ \ \ \ \ \ h_0^{\varrho}(\widehat{\alpha}_0;\phi) = \langle \widehat{\alpha}_0,\phi\rangle - \langle \varrho, e^{\phi}-1\rangle_{L^2(V)}.
    \end{equation}
Here the supremum is taken over all $\phi \in C^\infty(V)$.
    
The term $\mathcal{I}_{tr}(\widehat{\alpha})$ in \eqref{rate function} is the dynamic part of the rate function, and depends on the whole trajectory $\widehat{\alpha}$. It is given by the following:
    \begin{equation}\label{dynamic ldp}
        \mathcal{I}_{tr}(\widehat{\alpha}) = \sup_G \mathcal{I}_{tr}(\widehat{\alpha};G),
        \qquad \mathcal{I}_{tr}(\widehat{\alpha};G) = \ell(\widehat{\alpha}; G) -  \int_0^T \left< \widehat{\alpha}_t, c(\sigma, m(\widehat{\alpha}_t)) \left(e^{-\sigma \widetilde{G}_t(x)} -1\right)\right>\dd t.
    \end{equation}
Here the supremum is taken over all $G\in C^\infty([0,T]\times V)$, and we recall the notation $\widetilde{G}_t(x) = G_t(x,1) - G_t(x,-1)$. Furthermore, $\ell(\widehat{\alpha};G)$ is a linear functional, defined as follows
    \begin{equation}\label{ell}
        \ell(\widehat{\alpha};G) := \langle \widehat{\alpha}_T, G_T\rangle - \langle \widehat{\alpha}_0, G_0\rangle - \int_0^T \langle \widehat{\alpha}_t, (\partial_t + \sigma \partial_x) G_t\rangle \dd t,
    \end{equation}
and $m(\widehat{\alpha}_t)$ is the magnetization of the measure, defined as 
    \begin{equation}
        m(\widehat{\alpha}_t) = \frac{\langle\widehat{\alpha}_t,\mathds{1}_{\sigma=1} - \mathds{1}_{\sigma=-1}\rangle  }{\langle \widehat{\alpha}_t, 1\rangle}.
    \end{equation}
\begin{Remark}
    In Lemmas \ref{initial ldp good} and \ref{corollary ldp} we derive more explicit expressions of the static part $h_0^\varrho(\widehat{\alpha}_0)$ and the dynamic part $\mathcal{I}_{tr}(\widehat{\alpha})$. Specifically, we identify the functions $\phi$ and $H$ for which $h_0^\varrho(\widehat{\alpha}_0) = h_0^\varrho(\widehat{\alpha}_0;\phi)$ and $\mathcal{I}_{tr}(\widehat{\alpha}) = \mathcal{I}_{tr}(\widehat{\alpha};H)$.
\end{Remark}
In order to prove the large deviation principle with rate function $\mathcal{I}$, we need to establish the upper and lower bound.
\begin{enumerate}[label=\roman*.,font=\itshape]
    \item \textbf{Upper bound}: For every closed set $\mathcal{C} \subset  D([0,T];\mathcal{M}_V)$ we have 
        \begin{equation}\label{upper bound def}
            \uplim_{N\to\infty} \frac{1}{N} \log \mathbb{P}_N^{\varrho}(\pi^N_{\cdot} \in \mathcal{C}) \leq -\inf_{\widehat{\alpha} \in \mathcal{C}} \mathcal{I}^{\varrho}(\widehat{\alpha})
        \end{equation}
    \item \textbf{Lower bound}: For every open set $\mathcal{O} \subset  D([0,T];\mathcal{M}_V)$ we have  
        \begin{equation}\label{lower bound def}
            \lowlim_{N\to\infty} \frac{1}{N} \log \mathbb{P}_N^{\varrho}(\pi^N_{\cdot} \in \mathcal{O}) \geq -\inf_{\widehat{\alpha} \in \mathcal{O}} \mathcal{I}^{\varrho}(\widehat{\alpha})
        \end{equation}
\end{enumerate}

\subsection{Radon-Nikodym derivatives}
The goal of this section is to obtain an explicit form for the Radon-Nikodym derivative of the path-space measure of the weakly perturbed process $\dd\p^{\varrho,H}_N$ with respect to the path-space measure of the original process $\dd \p^{\varrho}_N$. This Radon-Nikodym derivative is given by the so-called exponential martingale of the process, and is given in the following lemma.
\begin{Lemma}
    For all $T>0$, $N\in\mathbb{N}$ and $H \in C^\infty(V)$, we have that 
        \begin{align}\label{exp mart}
            \frac{\dd\p^{\varrho,H}_N}{\dd\p^{\varrho}_N} 
            &= \mathcal{Z}^H_{N,T}(\pi^N_{\cdot})\nonumber\\
            &:= \exp\left(N\langle \pi^N_{T}, H_T\rangle  - N\langle \pi^N_{0}, H_0\rangle - \int_0^T e^{-N\langle \pi^N_{t},H_t\rangle}(\partial_t + \mathscr{L}_N) e^{N\langle \pi^N_{t},H_t\rangle} \dd t \right).
        \end{align}
\end{Lemma}
\begin{proof}
By Palmowski and Rolski \cite{palmowski}, the exponential martingale $\mathcal Z^H_{N,T}(\pi^N_{\cdot})$ is equal to the Radon-Nikodym derivative $\frac{\dd\widetilde{\p}_N}{\dd\p^{\varrho}_N}$, where $\widetilde{\p}_N$ is the path-space measure  (up to time $T$) of the process corresponding to the time-dependent Markov generator on $\widehat{\varrho}_N$ given by 
    \begin{equation}\label{palrol gen}
        \widetilde{\mathscr{L}}_{N,t}f(\eta) = e^{-N\langle \pi^N(\eta), H_t\rangle} \left[ \mathscr{L}_N\big( f(\eta)\cdot e^{N\langle \pi^N(\eta), H_t\rangle}\big) - f(\eta)\cdot \mathscr{L}_N e^{N\langle \pi^N(\eta), H_t\rangle}\right].
    \end{equation}
where $\pi^N$ now denotes the empirical measure as a function of $\eta \in \widehat{\varrho}_N$, as defined in \eqref{emp meas}. We will prove that $\widetilde{\p}_N = \p_{N}^{\varrho,H}$ by showing that $\widetilde{\mathscr{L}}_{N,t}$ is equal to the generator of the weakly perturbed process $\mathscr{L}_{N,t}^H$ as defined in \eqref{weak asym gen}. We compute, using \eqref{gen}
    \begin{align}\label{palrol1}
        e^{-N\langle \pi^N(\eta), H_t\rangle}&\mathscr{L}_N\big( f(\eta)\cdot e^{N\langle \pi^N(\eta), H_t\rangle}\big)\nonumber\\
        \nn=\ & N\sum_{(x,\sigma)\in V_N} \eta(x,\sigma)\big[f(\eta^{(x,\sigma)\to (x+\sigma,\sigma)})e^{H_t(\frac{x+\sigma}{N},\sigma) - H_t(\frac{x}{N},\sigma)}  - f(\eta)\big]\\
        &   + \sum_{(x,\sigma)\in V_N} c(\sigma,m(\eta)) \eta(x,\sigma)\big[f(\eta^{(x,\sigma)\to(x,-\sigma)})e^{H_t(\frac{x}{N},-\sigma) - H_t(\frac{x}{N},\sigma)}-f(\eta)\big],
    \end{align}
where we used that 
    \begin{equation}
        \langle X_N(\eta^{(x,\sigma)\to(y,\sigma')}), H_t\rangle - \langle X_N(\eta),H_t\rangle = H_t(\tfrac{y}{N},\sigma') - H_t(\tfrac{x}{N},\sigma).
    \end{equation}      
Similarly we compute
    \begin{align}\label{palrol2}
        e^{-N\langle \pi^N(\eta), H_t\rangle}\mathscr{L}_Ne^{N\langle \pi^N(\eta), H_t\rangle}(\eta)
        \nn=\ & N\sum_{(x,\sigma)\in V_N} \eta(x,\sigma)\big[e^{H_t(\frac{x+\sigma}{N},\sigma) - H_t(\frac{x}{N},\sigma)}  -1\big]\\
        &  + \sum_{(x,\sigma)\in V_N} c(\sigma,m(\eta)) \eta(x,\sigma)\big[e^{H_t(\frac{x}{N},-\sigma) - H_t(\frac{x}{N},\sigma)}-1\big].
    \end{align}
Substituting \eqref{palrol1} and \eqref{palrol2}  into \eqref{palrol gen}, we indeed find that $\widetilde{\mathscr{L}}_{N,t}=\mathscr{L}_{N,t}^H$, completing the proof. 
\end{proof}

\begin{Corollary}\label{radon good}
    \begin{equation}
        \frac{1}{N} \log \left(\mathcal{Z}_{N,T}^H(\pi^N_{\cdot}) \right)
        =  \mathcal{I}_{tr}(\pi^N_{\cdot};H)+ \mathcal{O}(\tfrac{1}{N}) .
    \end{equation}
\end{Corollary}
\begin{proof}
Making use of the following approximation
    \begin{equation}
        e^{H_t(\frac{x+\sigma}{N},\sigma) - H_t(\frac{x}{N},\sigma)}  -1 = \sigma \partial_xH_t(\tfrac{x}{N},\sigma) + \mathcal{O}(\tfrac{1}{N}),
    \end{equation}
we are able to write \eqref{palrol2} as
    \begin{equation}\label{hamiltonian}
        e^{-N\langle \pi^N_{t}, H_t\rangle}\mathscr{L}_Ne^{N\langle \pi^N_{t}, H_t\rangle} = N\langle \pi^N_{t}, \sigma \partial_xH_t\rangle  + N \left\langle \pi^N_{t}, c(\sigma, m(\pi^N_{t})) \left(e^{-\sigma \widetilde{H}_t(x)} -1\right)\right\rangle + \mathcal{O}(1). 
    \end{equation}
By now plugging this into \eqref{exp mart}, we find that 
    \begin{align}
        \mathcal{Z}_{N,T}^H(\pi^N_{\cdot}) 
        &\nn=\exp \left( N\langle \pi^N_{T}, H_T\rangle  - N\langle \pi^N_{0}, H_0\rangle - N\int_0^T \langle \pi^N_{t}, (\partial_t +\sigma \partial_x)H_t\rangle  \dd t + \mathcal{O}(1)\right)\\
        \nn&\ \ \ \ \ \ \ \times \exp\left( - N\int_0^T \left\langle \pi^N_{t}, c(\sigma, m(\pi^N_{t})) \left(e^{-\sigma \widetilde{H}_t(x)} -1\right)\right\rangle\dd t \right)\\
        &= \exp\left( N\ell(\pi^N_{\cdot};H) - N\int_0^T \left\langle \pi^N_{t}, c(\sigma, m(\pi^N_{t})) \left(e^{-\sigma \widetilde{H}_t(x)} -1\right)\right\rangle \dd t+ \mathcal{O}(1) \right),
    \end{align}
finishing the proof.
\end{proof}

\subsection{Upper bound}
In this section we will prove the large deviation upper bound \eqref{upper bound def}. A crucial ingredient is to prove that the path-space measures of $\pi^N_{\cdot}$ are exponentially tight, which then reduces the proof of \eqref{upper bound def} to compact sets.
\begin{Theorem}[Exponential Tightness] \label{exponential tightness}
For any $n\in\mathbb{N}$ there exists a compact set $\mathcal{K}_n\subset D([0,T],\mathcal{M}_V)$ such that 
    \begin{equation}\label{exp tight eq}
        \uplim_{N\to\infty} \frac{1}{N} \log \p^{\varrho}_N(\pi^N_{\cdot} \notin \mathcal{K}_n) = -n.
    \end{equation}
\end{Theorem}
Before proving  exponential tightness, we first give the proof of the upper bound for compact sets. 

\begin{Theorem}[Upper bound for compact sets] \label{upper bound compact}
    For every compact set $\mathcal{K} \subset  D([0,T];\mathcal{M}_V)$ we have that 
    \begin{equation}
        \uplim_{N\to\infty} \frac{1}{N} \log \mathbb{P}_N^{\varrho}(\pi^N_{\cdot} \in \mathcal{K}) \leq -\inf_{\widehat{\alpha} \in \mathcal{K}} \mathcal{I}^{\varrho}(\widehat{\alpha}).
    \end{equation}
\end{Theorem}
\begin{proof}
We start with the following computation
    \begin{align}\label{upper bound 1}
        \frac{1}{N} \log \mathbb{P}_N^{\varrho}(\pi^N_{\cdot} \in \mathcal{K})
        &= \frac{1}{N}\log \E_N^{\varrho}\left[\mathds{1}_{\pi^N_{\cdot} \in \mathcal{K}} \cdot \frac{e^{Nh_0^{\varrho}(\pi^N_{0};\phi)}}{e^{Nh_0^{\varrho}(\pi^N_{0};\phi)}}\cdot \frac{\mathcal{Z}^G_{N,T}(\pi^N_{\cdot})}{\mathcal{Z}^G_{N,T}(\pi^N_{\cdot})} \right]\\
        &\nn\leq -\inf_{\widehat{\alpha}\in\mathcal{K}} \frac{1}{N}\log \left[ e^{Nh_0^{\varrho}(\widehat{\alpha}_0;\phi)}\cdot \mathcal{Z}_{N,T}^G(\widehat{\alpha}) \right] + \frac{1}{N}\log \E^{\varrho}_N\left[ e^{Nh_0^{\varrho}(\pi^N_{0};\phi)}\cdot \mathcal{Z}^G_{N,T}(\pi^N_{\cdot}) \right],
    \end{align}
where we recall the definition of $h_0^\varrho$ in \eqref{rateini} and of $\mathcal{Z}^G_{N,T}$ in \eqref{exp mart}. Since $\mathcal{Z}^G_{N,T}(\pi^N_{\cdot})$ is a martingale with $\mathcal{Z}^G_{N,0}(\pi^N_{\cdot})=1$, we actually find that 
    \begin{align}
        \E^{\varrho}_N\left[e^{Nh_0^{\varrho}(\pi^N_{0};\phi)} \cdot \mathcal{Z}^G_{N,T}(\pi^N_{\cdot})\right] = \E_{\mu^{\varrho}_N}\left[e^{Nh_0^{\varrho}(\pi^N_{0};\phi)}\right] = \E_{\mu^{\varrho}_N}\left[e^{N\left(\langle \pi^N_{0},\phi\rangle - \langle\varrho,e^\phi-1\rangle_{L^2(V)}\right)}\right] = 1,
    \end{align}
where we used that $\mu^{\varrho}_N$ is a product Poisson distribution. Therefore the second term in \eqref{upper bound 1} vanishes. For the first term, note that we took $\phi$ and $G$ arbitrarily, so we have
    \begin{equation}
        \uplim_{N\to\infty} \frac{1}{N} \log \mathbb{P}_N^{\varrho}(\pi^N_{\cdot} \in \mathcal{K})
        \leq - \sup_{\phi,G}  \uplim_{N\to\infty} \inf_{\widehat{\alpha} \in \mathcal{K}}  \frac{1}{N}\log \left[ e^{Nh_0^{\varrho}(\widehat{\alpha}_0;\phi)}\cdot \mathcal{Z}_{N,T}^G(\widehat{\alpha}) \right] 
    \end{equation}
using Corollary \ref{radon good}, we find that 
    \begin{align}
        \nn- \sup_{\phi,G}  \uplim_{N\to\infty} \inf_{\widehat{\alpha}\in\mathcal{K}} \frac{1}{N}\log \left[ e^{Nh_0^{\varrho}(\widehat{\alpha}_0;\phi)}\cdot \mathcal{Z}_{N,T}^G(\widehat{\alpha}) \right]
        &= -\sup_{\phi,G} \inf_{\widehat{\alpha}\in\mathcal{K}} h_0^{\varrho}(\widehat{\alpha}_0;\phi) + \mathcal{I}_{tr}(\widehat{\alpha};G) \\
        &= -\inf_{\widehat{\alpha} \in \mathcal{K}}\mathcal{I}^{\varrho}(\widehat{\alpha}),
    \end{align}
where we were able to interchange the supremum over $\phi$ and $G$ with the infimum over $\widehat{\alpha}$ using the argument of Lemma 11.3 in \cite{varadhan1984large}, using that $\mathcal{K}$ is compact.  
\end{proof}

 In order to prove the exponential tightness, we want to use the method used in \cite[pages 271-273]{KipnisLandim}. However, since we do not know the invariant measure of this system, we first turn to a perturbed model of which we do know the invariant measures, namely the independent run-and-tumble particle system, as defined in \cite{Hidde-Fluctuations} (without diffusive jumps). The generator of this process corresponds to setting $c(\sigma,m) \equiv 1$ in the generator $\mathscr{L}_N$ defined in \eqref{gen}, and we will denote it by $\mathscr{L}_N^{\RTP}$. For this process, we know that the Product Poisson measures with constant density $\varrho_c$ are invariant. We denote the path-space measure of this process, started from $\mu^\varrho_N$, by $\p^{\RTP,\varrho}_{N}$. 
    \begin{Lemma}\label{radon bound}
    There exists a constant $C>0$ such that
        \begin{equation}\label{radon rtp}
            \E_{\p^{\RTP,\varrho}_N}\left[\left(\frac{\dd\p^{\varrho}_N}{\dd \p^{\RTP,\varrho_c}_N} \right)^2\right]\leq e^{CN}.
        \end{equation}
    \end{Lemma}
\begin{proof}
    Note that we can split the Radon-Nikodym derivative in the following way
        \begin{equation}
            \frac{\dd\p^{\varrho}_N}{\dd \p^{\RTP,\varrho_c}_N}  
            = \frac{\dd\mu^{\varrho}_N}{\dd \mu^{\varrho_c}_N}\cdot \frac{\dd\p^{\varrho}_N}{\dd \p^{\RTP,\varrho}_N}. 
        \end{equation}
    The Radon-Nikodym derivative of the Poisson distributions can be bounded in the following way
        \begin{align}\label{radonmu}
            \frac{\dd\mu^{\varrho}_N}{\dd \mu^{\varrho_c}_N}(\eta) 
            &\nn= \exp\left( \sum_{(x,\sigma)\in V_N} \eta(x,\sigma) \log\left(\frac{\varrho(\tfrac{x}{N},\sigma)}{\varrho_c}\right) - \sum_{(x,\sigma)\in V_N} (\varrho(\tfrac{x}{N},\sigma) - \varrho_c)\right)\\
            &\leq \exp\left(\log\left( \frac{||\varrho||_\infty}{\varrho_c}\right)|\eta^N| + 2\varrho_c N \right).
        \end{align}
    Since the jump rates of the processes corresponding to $\p^{\RTP,\varrho}_N$ and $\p^{\varrho}_N$ only differ at the internal state jumps, by the Girsanov formula, we find that the Radon-Nikodym derivative is given by 
        \begin{align}
            \frac{\dd\p^{\varrho}_N}{\dd \p^{\RTP,\varrho}_N}  
            &= \exp\left(\sum_{(x,\sigma) \in V_N} \int_0^T \log\big(c(\sigma,m_{N}(t))\big)\dd J_t^{(x,\sigma) \to (x,-\sigma)} - \sum_{(x,\sigma) \in V_N} \int_0^T \eta^N_t(x,\sigma)\big( c(\sigma, m_N(t)) - 1 \big) \dd t\right),
        \end{align}
    where  $J_t^{(x,\sigma) \to (x,-\sigma)}$ is the number of jumps made from $(x,\sigma)$ to $(x,-\sigma)$ up to time $t$. Since $c(\sigma,m_N(t))$ is bounded from above and below, we can find constants $c_1, c_2>0$ such that 
        \begin{align}\label{radonP}
            \left(\frac{\dd\p^{\varrho}_N}{\dd \p^{\RTP,\varrho}_N} \right)^2
            \leq \exp\left(c_1\sum_{(x,\sigma)\in V_N} J_T^{(x,\sigma) \to (x,-\sigma)} + c_2 T|\eta^N| \right),
        \end{align}
    where we recall that $|\eta^N|$ is the total number of particles in the configuration $\eta^N$. Note that $\sum_{(x,\sigma)\in V_N} J_T^{(x,\sigma) \to (x,-\sigma)}$ is the total number of internal state jumps up to time $T$, which under $\p^{\RTP,\varrho}_N$ is a Poisson process with intensity $|\eta^N|$. Combining \eqref{radonmu} and \eqref{radonP}, we can find a constant $c_3>0$ such that 
        \begin{align}
            \E_{\p^{\RTP,\varrho}_N}\left[\left(\frac{\dd\p^{\varrho}_N}{\dd \p^{\RTP,\varrho_c}_N}\right)^2 \right]
            &\nn\leq e^{4\varrho_c N}\E_{\p^{\RTP,\varrho}_N}\left[\exp\left(c_3T |\eta^N|\right) \right]\\
            &\leq \exp\left(4\varrho_cN+ 2||\varrho||_\infty \left(e^{c_3T}-1\right) N \right),
        \end{align}
    hence \eqref{radon rtp} holds.
\end{proof}
Now we can proceed with the proof of exponential tightness, using the approach from \cite{KipnisLandim}. We will need the following result.
    \begin{Lemma}\label{cont cond}
        For every $\varepsilon>0$ and $G \in C^\infty(V)$ 
            \begin{equation}\label{Arzela Ascoli}
                \lim_{\delta\to0} \lim_{N\to\infty} \frac{1}{N}\log \p^{\varrho}_N\left(\sup_{|s-t|<\delta} \left| \left<\pi^N_{t}, G\right> - \left<\pi^N_{s}, G\right> \right| \geq \varepsilon \right)= -\infty.
            \end{equation}
    \end{Lemma}
\begin{proof}
Let $\varepsilon>0$ be given. First note that we have the following 
    \begin{align}
        \p^{\varrho}_N\left(\sup_{|s-t|<\delta} \left| \left<\pi^N_{t}, G\right> - \left<\pi^N_{s}, G\right> \right| \geq \varepsilon \right)
        \nn&\leq \p^{\varrho}_N\left(\sup_{|s-t|<\delta}  \left<\pi^N_{t}, G\right> - \left<\pi^N_{s}, G\right>  \geq \varepsilon \right)\\
        &\qquad +\p^{\varrho}_N\left(\sup_{|s-t|<\delta}  \left<\pi^N_{t}, -G\right> - \left<\pi^N_{s}, -G\right>  \geq \varepsilon \right).
    \end{align}
Since we are considering every $G\in C^\infty(V)$, we can neglect the absolute value in \eqref{Arzela Ascoli}.
Furthermore, by H\"older's inequality we have that for a general event $A$
    \begin{equation}
        \p^{\varrho}_N(A) 
        = \E_{\p^{\RTP,\varrho_c}_N}\left[\mathds{1}_A \frac{\dd\p^{\varrho}_N}{\dd\p^{\RTP,\varrho_c}_N}\right] 
        \leq \left(\p^{\RTP,\varrho_c}_N(A)\right)^{\frac{1}{2}} \left(\E_{\p^{\RTP,\varrho}_N}\left[\left(\frac{\dd\p^{\varrho}_N}{\dd \p^{\RTP,\varrho_c}_N}\right)^2 \right]\right)^{\frac{1}{2}}.
    \end{equation}
Therefore, by Lemma \ref{radon bound}, it is enough to prove the result for $\p^{\RTP,\varrho_c}_N$.

Using the following inclusion.
    \begin{equation}\label{contain}
        \left\{ \sup_{|s-t|<\delta} \langle\pi^N_{t}, G\rangle - \langle\pi^N_{s}, G\rangle \geq \varepsilon\right\} \subset \bigcup_{k=0}^{[T\delta^{-1}]} \left\{\sup_{0\leq t<\delta } \langle\pi^N_{k\delta + t}, G\rangle - \langle\pi^N_{k\delta}, G\rangle >\frac{\varepsilon}{4}\right\},
    \end{equation}   
we are able to find that 
    \begin{align}
        &\nn\uplim_{N\to\infty} \frac{1}{N}\log \p^{\RTP,\varrho_c}_N\left(\sup_{|s-t|<\delta} \langle\pi^N_{t}, G\rangle - \langle\pi^N_{s}, G\rangle  \geq \varepsilon \right)\\
        \nn&\qquad \qquad \leq  \uplim_{N\to\infty} \frac{1}{N}\log \max_{k=0}^{[T\delta^{-1}]} \p^{\RTP,\varrho_c}_N\left(\sup_{0\leq t<\delta } \langle\pi^N_{k\delta + t}, G\rangle - \langle\pi^N_{k\delta}, G\rangle  \geq \tfrac{1}{4}\varepsilon \right)\\
        &\qquad \qquad = \uplim_{N\to\infty} \frac{1}{N}\log  \p^{\RTP,\varrho_c}_N\left(\sup_{0\leq t<\delta }  \langle\pi^N_{t}, G\rangle - \langle\pi^N_{0}, G\rangle \geq \tfrac{1}{4}\varepsilon \right),
    \end{align}
where in the last step we used  that $\mu^{\varrho_c}_N$ is invariant for the RTP system. 
By denoting the exponential martingale corresponding to $\mathscr{L}^{\RTP}_N$ as $\mathcal{Z}^{\RTP, G}_{N,t}(\pi^N_{\cdot})$ (recall the definition of the exponential martingale in \eqref{exp mart}), and multiplying both sides by a constant $\lambda>0$, we find that
    \begin{align}
        &\nn\p^{\RTP,\varrho_c}_N\left(\sup_{0\leq t<\delta }  \langle\pi^N_{ t}, \lambda G\rangle - \langle\pi^N_{0}, \lambda G\rangle\geq \tfrac{1}{4}\lambda \varepsilon \right)\\
        &\nn \qquad \leq \p^{\RTP,\varrho_c}_N\left(\sup_{0\leq t<\delta } \frac{1}{N}\log \mathcal{Z}^{\RTP, \lambda G}_{N,t}(\pi^N_{\cdot}) + \frac{1}{N}\int_{0}^{t} e^{-N\langle \pi^N_{s},\lambda G_s\rangle}(\partial_s + \mathscr{L}^{\RTP}_N) e^{N\langle \pi^N_{s},\lambda G_s\rangle} \dd s \geq \tfrac{1}{4}\lambda \varepsilon \right)\\
        \nn&\qquad \leq \p^{\RTP,\varrho_c}_N\left(\sup_{0\leq t<\delta } \mathcal{Z}^{\RTP,\lambda G}_{N,t}(\pi^N_{\cdot}) \geq e^{\frac{1}{8}N\lambda \varepsilon}\right)\\
        &\qquad \qquad \qquad +\p^{\RTP,\varrho_c}_N\left( \sup_{0\leq t < \delta}\frac{1}{N}\int_{0}^{ t} e^{-N\langle \pi^N_{s},\lambda G_s\rangle}(\partial_s + \mathscr{L}^{\RTP}_N) e^{N\langle \pi^N_{s},\lambda G_s\rangle} \dd s \geq \tfrac{1}{8}\lambda \varepsilon \right).
    \end{align}
Recalling that $\mathcal{Z}^{\RTP,\lambda G}_{N,t}(\pi^N_{\cdot})$ is a non-negative martingale, by  Doob's martingale inequality we can upper bound the first part by 
    \begin{equation}\label{rtp 1}
        \p^{\RTP,\varrho_c}_N\left(\sup_{0\leq t<\delta } \mathcal{Z}^{\RTP,\lambda G}_{N,t}(\pi^N_{\cdot}) \geq e^{\frac{1}{8}N\lambda \varepsilon}\right) \leq  \E^{\RTP,\varrho_c}_N\left[ \mathcal{Z}^{\RTP,\lambda G}_{N,\delta}\right]\cdot e^{-\frac{1}{8}N\lambda \varepsilon}  = e^{-\frac{1}{8}N\lambda \varepsilon}. 
    \end{equation}
For the second part, using \eqref{hamiltonian}, we are able to find the following upper bound for the integrand
    \begin{align}
        \frac{1}{N} e^{-N\langle \pi^N_{s},\lambda G_s\rangle}(\partial_s + \mathscr{L}^{\RTP}_N) e^{N\langle \pi^N_{s},\lambda G_s\rangle}
        \leq \frac{1}{N}|\eta^N|\left( \lambda ||\partial_x G||_\infty  + M\left(e^{\lambda ||\widetilde{G}||_\infty}-1\right)\right).
    \end{align}
Therefore, using that $|\eta^N|$ is Poisson distributed with parameter $N\varrho_c$ under $\p^{\RTP,\varrho_c}_N$, we find that by the Markov inequality,
    \begin{align}\label{rtp 2}
        &\p^{\RTP,\varrho_c}_N\left( \sup_{0\leq t < \delta}\frac{1}{N}\int_{0}^{ t} e^{-N\langle \pi^N_{s},\lambda G_s\rangle}(\partial_s + \mathscr{L}^{\RTP}_N) e^{N\langle \pi^N_{s},\lambda G_s\rangle} \dd s \geq \tfrac{1}{8}\lambda \varepsilon \right) = \mathcal{O}(\delta).
    \end{align}
Combining \eqref{rtp 1} and \eqref{rtp 2}, we find that 
    \begin{equation}
        \lim_{\delta\to0} \lim_{N\to\infty} \frac{1}{N}\log \p^{\RTP,\varrho_c}_N\left(\sup_{|s-t|<\delta} \left| \left<\pi^N_{t}, G\right> - \left<\pi^N_{s}, G\right> \right| \geq \varepsilon \right)\leq -\frac{1}{8}\lambda \varepsilon,
    \end{equation}
and since we can take $\lambda$ arbitrarily large, this concludes the proof.
\end{proof}

We are now ready to give a proof of the exponential tightness.
\begin{proof}[Proof of Theorem \ref{exponential tightness}]
We start by defining the following set 
    \begin{equation}
        \mathcal{E}_K = \left \{\widehat{\alpha} \in D([0,T];\mathcal{M}_V): \sup_{t\in[0,T]} \widehat{\alpha}_t(V) \leq K\right \}.
    \end{equation}
For this set, by the Chernoff inequality, we find that  
    \begin{equation}\label{tip 1}
        \p^{\varrho}_N(\pi^N_{\cdot} \notin \mathcal{E}_K) = \p^{\varrho}_N\left(\sup_{t\in[0,T]} \pi^N_{t}(V) > K\right) \leq e^{-NK} \E^{\varrho}_N \left[\exp\left(\sup_{t\in[0,T]} N\pi^N_{t}(V)\right) \right],
    \end{equation}
where the expectation can be upper bounded in the following way
    \begin{align}\label{tip 2}
        \E^{\varrho}_N \left[\exp\left(\sup_{t\in[0,T]} N\pi^N_{t}(V)\right) \right] = \E^{\varrho}_N \left[\exp\left(|\eta^N|\right)\right] \leq e^{N||\varrho||_\infty(e-1)}.
    \end{align}
Combining \eqref{tip 1} and \eqref{tip 2}, we can find a sequence of numbers $(K_n)_{n\in\mathbb{N}}$ such that for every $n\in\mathbb{N}$
    \begin{equation}
        \uplim_{N\to\infty} \frac{1}{N}\log\p^{\varrho}_N(\pi^N_{\cdot} \notin \mathcal{E}_{K_n}) \leq -n
    \end{equation}

Next we consider a countable uniformly dense family $\{\phi_j\}_{j\in\mathbb{N}} \subset C^\infty(V)$ and define for each $\delta>0$ and $\varepsilon>0$ the following set
    \begin{equation}
        \mathcal{C}_{j,\delta,\varepsilon} = \left\{\widehat{\alpha} \in D([0,T];\mathcal{M}_V) : \sup_{|t-s|<\delta} \left|\langle \widehat{\alpha}_t, \phi_j\rangle - \langle \widehat{\alpha}_s,\phi_j\rangle\right| \leq \varepsilon \right\}.
    \end{equation}
Note that for any choice of the parameters the set $\mathcal{C}_{j,\delta,\varepsilon}$ is closed, and by Lemma \ref{cont cond} there exists a $\delta = \delta(j,m,n)$ such that 
    \begin{equation}
        \p^\varrho_N(\pi^N_{\cdot} \notin \mathcal{C}_{j,\delta,1/m} ) \leq \exp(-Nnmj)
    \end{equation}
for large enough $N$, and so 
    \begin{equation}
        \p^\varrho_N\left(\pi^N_{\cdot} \notin \bigcap_{j\geq 1, m\geq 1} \mathcal{C}_{j,\delta(j,m,n),1/m}  \right) \leq \sum_{j\geq 1, m\geq 1} \exp(-Nnmj) \leq C \exp(-Nn),
    \end{equation}
for some constant $C>0$. By now considering the set
    \begin{equation}
        \mathcal{K}_n =  \mathcal{E}_K \cap \bigcap_{j\geq 1, m\geq 1} \mathcal{C}_{j,\delta(j,m,n),1/m},
    \end{equation}
it follows that \eqref{exp tight eq} holds for this choice of $\mathcal{K}_n$. Since we furthermore know that it is closed we only need to show that it is relatively compact, which can be done by proving the following two things \cite[Proposition 4.1.2]{KipnisLandim}:
    \begin{enumerate}
        \item\label{rel comp 1} $\{\widehat{\alpha}_t : \widehat{\alpha} \in \mathcal{K}_n, t\in[0,T]\}$ is relatively compact in $\mathcal{M}_V$. 
        \item\label{rel comp 2} $\lim_{\delta\to 0} \sup_{\widehat{\alpha} \in \mathcal{K}_n} w_\delta(\widehat{\alpha}) = 0$, where 
            \begin{equation}
                w_\delta(\widehat{\alpha} ) := \sup_{|t-s|\leq \delta} \sum_{k=1}^\infty \frac{1}{2^k}\left( 1\wedge |\langle \widehat{\alpha}_t, \phi_j\rangle - \langle \widehat{\alpha}_s, \phi_j \rangle |\right)=0.
            \end{equation}
    \end{enumerate}
Here item \ref{rel comp 1}. is satisfied since $\mathcal{K}_n \subset \mathcal{E}_{K_n}$  and closed balls are compact in $\mathcal{M}_V$, and \ref{rel comp 2}. follows from the definition of the sets $C_{j,\delta,\varepsilon}$.
\end{proof}

\subsection{Lower bound}
In this section we will prove the large deviation lower bound, as given in \eqref{lower bound def}. The main idea is to show that if $\mathcal{I}^{\varrho}(\widehat{\alpha})<\infty$,  then there exists a function $H$ such that 
    \begin{equation}
        \mathcal{I}^{\varrho}(\widehat{\alpha}) = \lim_{N\to\infty} \frac{1}{N}\E^{\widehat{\alpha}_0}_N \left[\frac{\dd \p^{\widehat{\alpha}_0,H}_N}{\dd \p^{\varrho}_N}\right].
    \end{equation}
To achieve this, we have to show two things. First we have to show that if $h^{\varrho}_0(\widehat{\alpha}_0)<\infty$, then $\widehat{\alpha}_0$ has a density $\widehat{\varrho}_0$, and $h^{\varrho}_0$ can be written as the relative entropy of product Poisson distributions with the respective densities $\varrho$ and $\widehat{\varrho}_0$. After that, we have to show that if $\mathcal{I}_{tr}(\widehat{\alpha})<\infty$, then there exists a measurable $H:[0,T]\times V \to \mathbb{R}$ such that $\widehat{\alpha}$ satisfies the hydrodynamic equation of the weakly perturbed model, given in \eqref{weak-PDE}, and that $H$ is then a function for which the supremum in the definition of $\mathcal{I}_{tr}$ in \eqref{dynamic ldp} is attained.

The first step follows from the following Lemma
\begin{Lemma}\label{initial ldp good}
    If $h_0^{\varrho}(\widehat{\alpha}_0)<\infty$, then $\widehat{\alpha}_0$ has a density $\widehat{\varrho}_0:V\to\mathbb{R}$, and 
        \begin{equation}\label{static lemma}
            h_0^{\varrho}(\widehat{\alpha}_0) = \lim_{N\to\infty} \frac{1}{N}\E_{\mu^{\widehat{\varrho}_0}_N}\left[\log\frac{\dd \mu^{\widehat{\varrho}_0}_N}{\dd \mu^{\varrho}_N}\right].
        \end{equation}
\end{Lemma}
\begin{proof}
    Assume $h_0^{\varrho}(\widehat{\alpha}_0)<\infty$ If $\widehat{\alpha}_0$ is not absolutely continuous, then there exists a Borel set $A \subset V$ such that $\widehat{\alpha}_0(A) > 0$ and $\lambda_V(A)=0$, with $\lambda_V$ the Lebesgue measure on $V$. By the definition of $h_0^\varrho(\widehat{\alpha}_0)$ in \eqref{rateini}, we have that for $\phi \in C^\infty(V)$
        \begin{equation}\label{gamma0upp}
             \langle \widehat{\alpha}_0, \phi \rangle \leq h_0^\varrho(\widehat{\alpha}_0) + \langle \varrho, e^\phi -1 \rangle_{L^2(V)}.
        \end{equation}
    For $n\in\mathbb{N}$ now take a sequence $(\phi^{(n)}_k)_{k\in\mathbb{N}} \subset C^\infty(V)$ such that $\phi^{(n)}_k \to n\mathds{1}_A$ pointwise as $k\to\infty$. It then follows that
        \begin{equation}
            \langle \widehat{\alpha}_0, \phi^{(n)}_k \rangle\to n\widehat{\alpha}_0(A), \qquad \langle \varrho, e^{\phi^{(n)}_k} -1 \rangle_{L^2(V)} \to 0,
        \end{equation} 
    as $k\to\infty$. By taking $n$ large enough this contradicts \eqref{gamma0upp},
    hence we can conclude that  $\widehat{\alpha}_0$ has a density $\widehat{\varrho}_0$.

    The rest of the proof of \eqref{static lemma} follows then from calculating the supremum.
    \begin{align}
        \nn h^{\varrho}_0(\widehat{\alpha}_0) 
        &= \sup_\phi \left\{\langle \widehat{\varrho}_0, \phi\rangle_{L^2(V)} - \langle \varrho, e^\phi -1 \rangle_{L^2(V)}\right\}\\
        \nn&= \langle \widehat{\varrho}_0, \log(\tfrac{\widehat{\varrho}_0}{\varrho})\rangle_{L^2(V)} - \langle \widehat{\varrho}_0 - \varrho, 1\rangle_{L^2(V)}\\
        &= \lim_{N\to\infty} \frac{1}{N}\E_{\mu^{\widehat{\varrho}_0}_N}\left[\log\frac{\dd \mu^{\widehat{\varrho}_0}_N}{ \dd \mu^{\varrho}_N}\right].
    \end{align}
    where the supremum is attained for $\phi = \log(\frac{\widehat{\varrho}_0}{\varrho})$.
\end{proof}
By a similar argument as in the previous lemma, we can show that if $\mathcal{I}_{tr}(\widehat{\alpha})<\infty$ then there exists a density $\widehat{\varrho}: [0,T]\times V\to\mathbb{R}$ for the whole trajectory $\widehat{\alpha}$. In the following Lemma we prove an alternative formula for the dynamic part of the rate function in the case that $\widehat{\varrho}_t(v)>0$ for all $t\in[0,T], v\in V$.
\begin{Lemma}\label{corollary ldp}
     If $\mathcal{I}_{tr}(\widehat{\alpha})<\infty$ and $\widehat{\varrho}>0$ then there exists a bounded measurable function $H$ such that $\widehat{\alpha}$ satisfies the equation \eqref{weak-PDE} in the weak sense. Furthermore, $\mathcal{I}_{tr}(\widehat{\alpha}) = \mathcal{I}_{tr}(\widehat{\alpha};H)$ and
        \begin{equation}\label{exact formula rate}
            \mathcal{I}_{tr}(\widehat{\alpha}) = \int_0^T  \left\langle \widehat{\alpha}_t, \left( e^{-\sigma  \widetilde{H}_t(x)}(-\sigma  \widetilde{H}_t(x) -1) +1\right) c(\sigma, m(\widehat{\alpha}_t))\right\rangle \dd t.
        \end{equation}
\end{Lemma}
\begin{proof}
By the definition of $\mathcal{I}_{tr}$ in \eqref{dynamic ldp}, we have the following 
    \begin{equation}
        \sup_{\substack{G \in C^\infty([0,T]\times V)\\||G||_\infty\leq 1}} \left\{\ell(\widehat{\alpha};G) - \int_0^T \left\langle \widehat{\alpha}_t, c(\sigma,m(\widehat{\alpha}_t))\left( e^{-\sigma \widetilde{G}_t(x)}-1\right)\right\rangle \dd t\right\}\leq \mathcal{I}_{tr}(\widehat{\alpha}),
    \end{equation}
and so
    \begin{equation}
         \sup_{\substack{G \in C^\infty([0,T]\times V)\\||G||_\infty\leq 1}} \ell(\widehat{\alpha};G) \leq \mathcal{I}_{tr}(\widehat{\alpha}) + \int_0^T \left\langle \widehat{\alpha}_t, c(\sigma,m(\widehat{\alpha}_t))\left( e-1\right)\right\rangle \dd t < \infty.
    \end{equation}
Consequently, by the Hahn-Banach theorem, we can extend the linear functional $\ell(\widehat{\alpha};\cdot)$ to a bounded linear functional in $C([0,T]\times V)$. Therefore, by the Riesz representation theorem, there exists a signed measure $\nu \in \mathcal{M}_{[0,T]\times V}$ such that 
    \begin{equation}
        \ell(\widehat{\alpha}; G) = \langle \nu, G\rangle := \int_{[0,T]\times V} G d\nu.
    \end{equation}
Again, since we assume that $\mathcal{I}_{tr}(\widehat{\alpha})<\infty$,  this measure $\nu$ has a density $g:[0,T]\times V \to \mathbb{R}$. By then plugging in the definition of $\ell(\widehat{\alpha};G)$ in \eqref{ell}, we find that $\widehat{\alpha}$ satisfies
    \begin{equation}
         \langle \widehat{\alpha}_T, G_T\rangle - \langle \widehat{\alpha}_0, G_0\rangle - \int_0^T \langle \widehat{\alpha}_t, (\partial_t + \sigma \partial_x) G_t\rangle \dd t = \int_0^T \langle g_t, G_t \rangle_{L^2(V)} \dd t,
    \end{equation}
i.e., it satisfies the following PDE in the weak sense
    \begin{equation}
        \dot{\widehat{\alpha}}_t(x,\sigma) = -\sigma \partial_x \widehat{\alpha}_t(x,\sigma) + g_t(x,\sigma).
    \end{equation}
We can now split up  $g_t(x,\sigma) = \sigma f_t(x) + h_t(x)$, and we will show that $h\equiv 0$ almost everywhere. To see this, note that 
    \begin{align}
        \mathcal{I}_{tr}(\widehat{\alpha})
        &\nn= \sup_G \left\{\int_0^T \langle \sigma f_t, G_t\rangle_{L^2(V)} \dd t + \int_0^T \langle h_t, G_t \rangle_{L^2(V)} \dd t -\int_0^T \left\langle \widehat{\alpha}_t, c(\sigma,m(\widehat{\alpha}_t))\left( e^{-\sigma \widetilde{G}_t(x)}-1\right)\right\rangle \dd t \right\}\\
        &= \sup_{\widetilde{G},\overline{G}} \left\{ \int_0^T \langle f_t, \widetilde{G}_t\rangle_{L^2(\mathbb{T})} \dd t + \int_0^T \langle h_t, \overline{G}_t \rangle_{L^2(\mathbb{T})} \dd t -\int_0^T \left\langle \widehat{\alpha}_t, c(\sigma,m(\widehat{\alpha}_t))\left( e^{-\sigma \widetilde{G}_t(x)}-1\right)\right\rangle \dd t\right\},
    \end{align}
where  $\overline{G}_t(x) = G_t(x,1) + G_t(x,-1)$. By considering functions where $G_t(x,1) = G_t(x,-1)$, it follows that if $h\not\equiv 0$ almost everywhere then the last supremum is infinite, which contradicts $\mathcal{I}^{\varrho}(\widehat{\alpha})<\infty$. Therefore, 
    \begin{equation}\label{flux}
        \dot{\widehat{\alpha}}_t(x,\sigma) = -\sigma \partial_x \widehat{\alpha}_t(x,\sigma) + \sigma f_t(x)
    \end{equation}
holds weakly, with $f:[0,T]\times \mathbb{T}\to\mathbb{R}$ some bounded measurable function.

Setting $\Psi_t(x,\sigma) = e^{-\sigma \widetilde{H}_t(x)}$ in \eqref{weak-PDE}, we solve the following equation for $\Psi_t$, 
\begin{equation}\label{WA Psi}
\sigma f_t(x) =  \frac{1}{\Psi_t(x,\sigma)}  c(-\sigma, m(\widehat{\alpha}_t))\widehat{\varrho}_t(x,-\sigma) -\Psi_t(x,\sigma) c(\sigma, m(\widehat{\alpha}_t))\widehat{\varrho}_t(x,\sigma)
\end{equation}
This is a quadratic equation in $\Psi_t$, with the following positive and bounded solution
\begin{align}\label{Psi}
\Psi_t(x,\sigma) = \frac{-\sigma f_t(x) + \sqrt{f_t(x)^2 + 4\widehat{\varrho}_t(x,\sigma) c(\sigma,m(\widehat{\alpha}_t))\widehat{\varrho}_t(x,-\sigma) c(-\sigma,m(\widehat{\alpha}_t))}}{2\widehat{\varrho}_t(x,\sigma) c(\sigma,m(\widehat{\alpha}_t))}.
\end{align}
It is a straightforward calculation to show that $\Psi(x,\sigma) \cdot \Psi(x,-\sigma) = 1$, and so \eqref{weak-PDE} holds for $\widetilde{H}_t(x) = -\log(\Psi(x,1))$. 

Now, assuming that $\widehat{\alpha}$ satisfies \eqref{weak-PDE} we find that 
 \begin{align}
    \mathcal{I}_{tr}(\widehat{\alpha})
    &\nn= \sup_G\left\{\ell(\widehat{\alpha};G) - \int_0^T \langle \widehat{\alpha}_t, \left(e^{-\sigma\widetilde{G}_t(x)}-1\right)c(\sigma, m(\widehat{\alpha}_t)) \rangle \dd t \right\}\\
    \nn&= \sup_G \left\{\int_0^T \left\langle  \widehat{\alpha}_t, \left(-e^{-\sigma \widetilde{H}_t(x)} \sigma \widetilde{G}_t(x)- e^{-\sigma \widetilde{G}_t(x)} + 1 \right) c(\sigma, m(\widehat{\alpha}_t))\right\rangle \dd t\right\} \\
    &= \int_0^T \Big\langle \widehat{\alpha}_t, \sup_{p\in\mathbb{R}} \left(-e^{-\sigma \widetilde{H}_t(x)} \sigma p- e^{-\sigma p} + 1 \right) c(\sigma, m(\widehat{\alpha}_t))\Big\rangle \dd t,
    \end{align}
where we can interchange the supremum in the last integral by dominated convergence. This supremum is attained for $p=\widetilde{H}_t(x)$, indeed showing  that $\mathcal{I}_{tr}(\widehat{\alpha}) = \mathcal{I}_{tr}(\widehat{\alpha};H)$. After filling this in, it follows that \eqref{exact formula rate} holds. 
\end{proof}

We now have a clear formulation of the rate function when $\widehat{\alpha}$ has positive density, however if the density can be zero then the formulation in \eqref{Psi} is not well-defined. Furthermore, in order for the hydrodynamic limit of the weakly perturbed model to hold in Theorem \ref{hydro theorem}, we need to assume that $H \in C^{1,0}([0,T]\times V)$. We therefore define the following space 
    \begin{equation}
        D_0([0,T];\mathcal{M}_V) := \left\{\widehat{\alpha} \in D([0,T];\mathcal{M}_V): \widehat{\varrho} >0 , \text{ $\widehat{\alpha}$ satisfies \eqref{weak-PDE} with $H\in C^{1,0}([0,T]\times V)$} \right\}
    \end{equation}
We will now show that the rate function of trajectories outside this set can be approximated by the rate function of trajectories within this set.

\begin{Lemma}\label{approx rate lemma}
    Let $\widehat{\alpha} \in D([0,T];\mathcal{M}_V)$ such $\mathcal{I}_{tr}(\widehat{\alpha})<\infty$, then there exists a sequence $(\widehat{\alpha}_k)_{k\in\mathbb{N}} \subset D_0([0,T];\mathcal{M}_V)$ such that $\widehat{\alpha}_k \to \widehat{\alpha}$ weakly and 
        \begin{equation}\label{approx rate}
            \mathcal{I}_{tr}(\widehat{\alpha}) = \lim_{k\to\infty} \mathcal{I}_{tr}(\widehat{\alpha}_k).
        \end{equation}
\end{Lemma}
\begin{proof}
    We first show that $\widehat{\alpha} \in D([0,T];\mathcal{M}_V)$ can be approximated  by trajectories with positive density. We define the following measure for any $\varepsilon>0$
        \begin{equation}
            \widehat{\alpha}_\varepsilon = (1-\varepsilon)\widehat{\alpha} + \varepsilon 1
        \end{equation}
    where $1$ on the right-hand side denotes the measure with constant density equal to 1. It follows that $\widehat{\alpha}_\varepsilon$ has positive density and that $\widehat{\alpha}_\varepsilon \to \widehat{\alpha}$ weakly as $\varepsilon\to 0$. Therefore, by convexity and lower semicontinuity of the rate function $\mathcal{I}^\varrho$, we then find that 
        \begin{equation}
            \uplim_{\varepsilon\to0} \mathcal{I}_{tr}(\widehat{\alpha}_\varepsilon) \leq \mathcal{I}_{tr}(\widehat{\alpha}) \leq \lowlim_{\varepsilon \to 0} \mathcal{I}_{tr}(\widehat{\alpha}_\varepsilon).
        \end{equation}
    Hence we have indeed found a good approximation.

    Now assume that $\widehat{\alpha} \in D([0,T];\mathcal{M}_V)$ has density $\widehat{\varrho}>0$ and that $\mathcal{I}_{tr}(\widehat{\alpha})<\infty$. By Lemma \ref{corollary ldp}, there exists a bounded measurable  $H:[0,T]\times V \to \mathbb{R}$ such that \eqref{weak-PDE} holds weakly. Now find a sequence $\widehat{\alpha}_k$ with densities $\widehat{\varrho}_k \in C^{2,1}([0,T]\times V)$ such that $\widehat{\varrho}_k\to \widehat{\varrho}$ pointwise as $k\to\infty$. It then follows that $\widehat{\alpha}_k \to \widehat{\alpha}$ weakly and, by \eqref{flux} and \eqref{Psi}, each $\widehat{\alpha}_k$ satisfies \eqref{weak-PDE} for some function  $H_k\in C^{1,0}([0,T]\times V)$ where $H_k\to H$ pointwise. By the formulation of $\mathcal{I}_{tr}(\widehat{\alpha})$ in \eqref{exact formula rate}, we can then indeed conclude that \eqref{approx rate} holds.
\end{proof}

\begin{Theorem}[Lower bound]\label{low bound neigh}
    Given a $\widehat{\alpha} \in D([0,T];\mathcal{M}_V)$, for every neighborhood $\mathcal{O} \subset D([0,T];\mathcal{M}_V)$ of $\widehat{\alpha}$ we have that 
        \begin{equation}
            \lowlim_{N\to\infty} \frac{1}{N} \log \mathbb{P}_N^{\varrho}(\pi^N_{\cdot} \in \mathcal{O}) \geq - \mathcal{I}^{\varrho}(\widehat{\alpha}).
        \end{equation}
    As a consequence, for every open set $\mathcal{O} \subset  D([0,T];\mathcal{M}_V)$ we have that 
        \begin{equation}
            \lowlim_{N\to\infty} \frac{1}{N} \log \mathbb{P}_N^{\varrho}(\pi^N_{\cdot} \in \mathcal{O}) \geq -\inf_{\widehat{\alpha} \in \mathcal{O}} \mathcal{I}^{\varrho}(\widehat{\alpha}).
        \end{equation}
\end{Theorem}
\begin{proof}
    If $\mathcal{I}^{\varrho}(\widehat{\alpha})=\infty$, the result is immediate, therefore we can assume that $\mathcal{I}^{\varrho}(\widehat{\alpha})<\infty$. By Lemma \ref{approx rate lemma} it is then enough to prove it for $\widehat{\alpha} \in D_0([0,T];\mathcal{M}_V)$, and so by Theorem \ref{weak hydro theorem} there exists an $H\in C^{1,0}([0,T]\times V)$ such that 
        \begin{equation}
            \p^{\widehat{\varrho}_0,H}_N(\pi^N_{\cdot} \in \cdot) \to \delta_{\widehat{\alpha}}.
        \end{equation}
    where $\widehat{\varrho}_0$ is the density of $\widehat{\alpha}_0$. Therefore, we have that 
        \begin{align}
            \lowlim_{N\to\infty} \frac{1}{N} \log \mathbb{P}_N^{\varrho}(\pi^N_{\cdot} \in \mathcal{O})
            &\nn= - \lowlim_{N\to\infty}\frac{1}{N}\log \E^{\widehat{\varrho}_0, H}_N\left[\mathds{1}_{\pi^N_{\cdot} \in \mathcal{O}} \frac{\dd \mu^{\widehat{\varrho}_0}_N}{\dd \mu^{\varrho}_N}\cdot\frac{\dd \p^{\widehat{\varrho}_0,H}_N}{\dd \p^{\widehat{\varrho}_0}_N} \right]\\
            \nn&\geq - \lowlim_{N\to\infty}\frac{1}{N}\E^{\widehat{\varrho}_0,H}_N\left[\log \frac{\dd \mu^{\widehat{\varrho}_0}_N}{\dd \mu^{\varrho}_N}\cdot\frac{\dd \p^{\widehat{\varrho}_0,H}_N}{\dd \p^{\widehat{\varrho}_0}_N} \right]\\
            &= - \mathcal{I}^{\varrho} (\widehat{\alpha}),
        \end{align}
    where we used Lemma \ref{initial ldp good}, Corollary \ref{radon good} and  that $\mathcal{I}_{tr}(\widehat{\alpha}) = \mathcal{I}_{tr}(\widehat{\alpha};H)$.
\end{proof}

\section{Proof of the hydrodynamic limit}\label{proof hydro}
In this section we prove the hydrodynamic limit of the weakly perturbed process defined in Section \ref{wam section}. We first prove that the PDE given in \eqref{weak-PDE} is well-posed, and afterwards we prove Theorem \eqref{weak hydro theorem}.
\subsection{Well-posedness of the PDE}
The approach is to define a sequence of densities $\varrho^{(n)}$ through the following recursive relation
\begin{multline}\label{iterative pde}
        \partial_t \varrho^{(n+1)}_t(x,\sigma) = -\sigma \partial_x\varrho^{(n+1)}_t(x,\sigma) + c(-\sigma,m(\varrho^{(n)}_t))e^{\sigma H_t(x)}\varrho^{(n+1)}_t(x,-\sigma)\\ - c(\sigma,m(\varrho^{(n)}_t))e^{-\sigma H_t(x)}\varrho^{(n+1)}_t(x,\sigma),
\end{multline}
where every $\varrho^{(n)}$ starts from $\varrho^{(n)}_0 = \varrho$. 
Setting $\mathcal{T}(\varrho^{(n)}) = \varrho^{(n+1)}$, it is enough to show that $\mathcal{T}$ is a contraction (up to some  finite time $T>0$) in the space $L^{\infty}([0,T];L^1(V))$. First note that for every trajectory $\varrho^{(n)}$, the trajectory $\varrho^{(n+1)}$ solving \eqref{iterative pde} satisfies the conservation of particles, hence for any $t\geq0$ we have that $||\varrho^{(n+1)}_t||_{L^1(V)} = ||\varrho||_{L^1(V)}$ with $\varrho$ the initial profile. Therefore, we indeed have that $\mathcal{T}: L^{\infty}([0,T];L^1(V))\to L^{\infty}([0,T];L^1(V))$. Now let $\varrho_1$, $\varrho_2$ be any two trajectories of densities, and denote the following 
\begin{align}
    &\nonumber\psi_1 = \mathcal{T}(\varrho_1), && \psi_2= \mathcal{T}(\varrho_2),\\
    &m_{1,t} = m(\varrho_{1,t}), &&m_{2,t} = m(\varrho_{2,t}).
\end{align}
Furthermore, denote $\delta = \psi_1 - \psi_2$ and $\varepsilon = m_1 - m_2$. For $[m]=\{m_t:t\in[0,T]\}$ a deterministic trajectory of magnetizations, we define the mapping 
\begin{equation}
    Q_{[m]}[\psi](x,\sigma,t) = c(-\sigma,m_t)e^{\sigma H_t(x)}\psi_t(x,-\sigma) - c(\sigma,m_t)e^{-\sigma H_t(x)}\psi_t(x,\sigma).
\end{equation}
Note that this mapping is linear in $\psi$. We then have that 
\begin{equation}
    \partial_t \delta_t(x,\sigma) = -\sigma \partial_x\delta_t(x,\sigma) + Q_{[m_1]}[\psi_1](x,\sigma,t) - Q_{[m_2]}[\psi_2](x,\sigma,t).
\end{equation}
We can rewrite the difference of the last two terms in a linear and non-linear part as follow 
\begin{equation}
    Q_{[m_1]}[\psi_1]- Q_{[m_2]}[\psi_2]
    = Q_{[m_1]}[\psi_1 - \psi_2] + \left(Q_{[m_1]}[\psi_2] - Q_{[m_2]}[\psi_2] \right).
\end{equation}
For the linear part, we have that 
\begin{align}\label{ping}
    \nonumber||Q_{[m_1]}[\delta](\cdot,\cdot,t)||_{L^1(V)} 
    &= \left|\left|  c(-\sigma,m_{1,t})e^{\sigma H_t(x)}\delta_t(x,-\sigma) - c(\sigma,m_{1,t})e^{-\sigma H_t(x)}\delta_t(x,\sigma)\right|\right|_{L^1(V)}\\
    &\leq C_1||\delta_t||_{L^1(V)},
\end{align}
with $C_1$ some constant depending on the bounded functions $c$ and $H$. For the non-linear part, we have that 
\begin{align}
    \nonumber&Q_{[m_1]}[\psi_2](x,\sigma,t) - Q_{[m_2]}[\psi_2](x,\sigma,t)\\
    &= [c(-\sigma,m_{1,t}) - c(-\sigma,m_{2,t})]e^{\sigma H_t(x)}\psi_{2,t}(x,-\sigma) - [c(\sigma,m_{1,t}) - c(\sigma,m_{2,t})]e^{-\sigma H_t(x)}\psi_{2,t}(x,\sigma).
\end{align}
By the Lipschitz-continuity of $c(\sigma,m)$ we now have that 
\begin{equation}\label{pong}
    ||Q_{[m_1]}[\psi_2](\cdot,\cdot,t) - Q_{[m_2]}[\psi_2](\cdot,\cdot,t)||_{L^1(V)} \leq L|\varepsilon_t|\cdot C_2||\psi_{2,t}||_{L^1(V)},
\end{equation}
with $L$ the Lipschitz-constant and $C_2$ some constant depending on $H$. Here $||\psi_{2,t}||_{L^1(V)} = ||\varrho||_{L^1(V)}$  by conservation of particles. Furthermore, we have that 
\begin{align}
    |\varepsilon_t| 
    = |m_{1,t}-m_{2,t}|
    = \frac{1}{||\varrho||_{L^1(V)}} \cdot ||\varrho_{1,t} - \varrho_{2,t}||_{L^1(V)}.
\end{align}
From \eqref{ping} and \eqref{pong}, we find that 
\begin{align}
    \nonumber\partial_t||\delta_t||_{L^1(V)} 
    &\leq ||Q_{[m_1]}[\delta](\cdot,\cdot,t)||_{L^1(V)}  +   ||Q_{[m_1]}[\psi_2](\cdot,\cdot,t) - Q_{[m_2]}[\psi_2](\cdot,\cdot,t)||_{L^1(V)} \\
    \nonumber
    &\leq C_1||\delta_{t}||_{L^1(V)} + LC_2||\varrho_{1,t} - \varrho_{2,t}||_{L^1(V)}\\
    &\leq C_1||\delta_{t}||_{L^1(V)} + LC_2||\varrho_{1} - \varrho_{2}||_{L^\infty([0,T];L^1(V))}.
\end{align}
Note that the transport term vanishes since $\int_{\mathbb{T}}\partial_x|\delta_t(x,\sigma)|dx =0$ by periodic boundary conditions. We can rewrite this in integral form as
\begin{equation}
    ||\delta_t||_{L^1(V)}  \leq \int_0^t C_1 ||\delta_s||_{L^1(V)}\dd s + LC_2t||\varrho_{1} - \varrho_{2}||_{L^\infty([0,T];L^1(V))}.
\end{equation}
By Gronwall's inequality, we find that 
\begin{equation}
    ||\delta_t||_{L^1(V)} \leq LC_2te^{ C_1t}||\varrho_{1} - \varrho_{2}||_{L^\infty([0,T];L^1(V))},
\end{equation}
and so we can conclude that 
\begin{equation}
    ||\mathcal{T}(\varrho_1) - \mathcal{T}(\varrho_2)||_{L^\infty([0,T];L^1(V))} \leq LC_2Te^{ C_1T}||\varrho_{1} - \varrho_{2}||_{L^\infty([0,T];L^1(V))} .
\end{equation}
Taking $T>0$ small enough, this is a contraction on $L^\infty([0,T];L^1(V))$, hence the sequence $\varrho^{(n)}$ converges locally to a unique solution of the PDE \eqref{weak-PDE}.

\subsection{Proof of Theorem \eqref{weak hydro theorem}}
Recall that $\alpha^H \in D([0,T];\mathcal{M}_V)$ denotes the process such that for every $t\in [0,T]$ the measure $\alpha^H_t$ has density  $\varrho^H_t(x,\sigma)$, which is the solution the the PDE given by \eqref{weak-PDE}.
In this way, $\alpha^H$ is the unique trajectory of measures measure such that for every $G \in C^\infty([0,T]\times V)$ and $t\in[0,T]$ we have that 
    \begin{equation}
        \mathscr{M}_t^{H,G}(\alpha^H) := \left<\alpha^H_t,G_t\right> - \left<\alpha^H_0,G_0\right>  - \int_0^t \left<\alpha^H_s, \left(\partial_s + \left(\!A^{H}_{s, \alpha^H}\!\right)^{\!*}\right) G_s\right>  \dd s = 0,
    \end{equation}
where $ \left(\!A^{H}_{s, \alpha^H}\!\right)^{\!*}$ is the differential operator given by 
    \begin{equation}
        \left(\!A^{H}_{s, \alpha^H}\!\right)^{\!*} G_s(x,\sigma) = \sigma \partial_x G_s(x,\sigma) + c(\sigma,m(\alpha^H_s))e^{-\sigma \widetilde{H}_s(x)}\big(G_s(x,-\sigma) - G_s(x,\sigma)\big), 
    \end{equation}
which is the action of the adjoint of $A^{H}_{s,\alpha^H}$ on smooth functions, with $A^{H}_{s,\alpha^H}$ given by
    \begin{equation}
         A^{H}_{s,\alpha^H}G(x,\sigma) =  -\sigma \partial_x G(x,\sigma) + c(-\sigma,m(\alpha^H_s))e^{\sigma \widetilde{H}_s(x)}G(x,-\sigma) - c(\sigma,m(\alpha^H_s))e^{-\sigma \widetilde{H}_s(x)}G (x,\sigma).
    \end{equation}
For a given $G \in C^\infty([0,T]\times V)$ we define the Dynkin Martingale
    \begin{equation}
        \mathscr{M}_{N,t}^{H,G} (\pi^N_{\cdot}) := \langle \pi^N_{t},G_t\rangle  - \langle \pi^N_{0},G_0\rangle  - \int_0^t (\partial_s +\mathscr{L}^H_{N,s} )\langle \pi^N_{s},G_s\rangle \dd s.
    \end{equation}

\begin{Lemma}\label{lemma 2.1}
For every $G \in C^\infty(V)$ we have
    \begin{equation}
        \lim_{N\to\infty} \E^{\varrho,H}_N\left[\sup_{t\in[0,T]} \left|\mathscr{M}_{N,t}^{H,G}(\pi^N_{\cdot}) - \mathscr{M}_t^{H,G}(\pi^N_{\cdot})\right|\right] = 0.
    \end{equation}
\end{Lemma}
\begin{proof}
Note that
    \begin{equation}
        \mathscr{M}_{N,t}^{H,G}(\pi^N_{\cdot}) - \mathscr{M}_t^{H,G}(\pi^N_{\cdot}) = \int_0^t\left( \mathscr{L}^H_{N,s} \langle \pi^N_{s},G_s\rangle  - \left\langle \pi^N_{s}, \left(\!A^{H}_{s, \pi^N_\cdot}\!\right)^{\!*} G_s\right\rangle\right) \dd s. 
    \end{equation}
So we need to calculate $ \mathscr{L}^H_{N,s} \langle \pi^N_{s},G_s\rangle $. In order to do that, we start with the preliminary calculation
    \begin{align}
        \mathscr{L}^H_{N,s} \langle \pi^N_{s},G_s\rangle  
        &= N\sum_{(x,\sigma) \in V_N} \eta^N_s(x,\sigma)e^{H_s(\frac{x+\sigma}{N},\sigma) - H_s(\frac{x}{N},\sigma)} \left[ \big\langle\pi^N\big((\eta^N_s)^{(x,\sigma)\to(x+\sigma,\sigma)}\big), G_s\big\rangle- \langle \pi^N_{s}, G_s \rangle  \right]\nonumber\\
        &\qquad  + \sum_{(x,\sigma) \in V_N} \eta_s^N(x,\sigma) c(\sigma,m(\pi^N_s)) e^{-\sigma \widetilde{H}_s(x)}\left[ \big\langle\pi^N\big((\eta^N_t)^{(x,\sigma)\to(x,-\sigma)}\big), G_s\big\rangle - \langle \pi^N_{s},G_s\rangle \right]\nonumber\\
        \nn&=   \sum_{(x,\sigma) \in V_N} \eta^N_s(x,\sigma) e^{H_s(\frac{x+\sigma}{N},\sigma) - H_s(\frac{x}{N},\sigma)}\big(G_s(\tfrac{x+\sigma}{N},\sigma) - G_s(\tfrac{x}{N},\sigma)\big)\\
        &\qquad + \frac{1}{N} \sum_{(x,\sigma)\in V_N} \eta_s^N(x,\sigma) c(\sigma,m(\pi^N_s)) e^{-\sigma \widetilde{H}_s(x)}\big(G_s(x,-\sigma)- G_s(x,\sigma)\big) ,
    \end{align}
 Using Taylor approximation, we can now write 
    \begin{align}\label{Ln}
        \nn \mathscr{L}^H_{N,s}\langle \pi^N_{s},G_s\rangle  &= \frac{1}{N} \sum_{(x,\sigma) \in V_N} \eta^N_s(x,\sigma) \left[\sigma\partial_xG_s(\tfrac{x}{N},\sigma)  + c(\sigma,m(\pi^N_s))e^{-\sigma \widetilde{H}_s(x)}\big(G_s(x,-\sigma)- G_s(x,\sigma)\big) \right]\\
        \nn&\qquad \qquad \qquad + R(N,H, G,s)\\
        &= \left\langle \pi^N_{s}, \left(\!A^{H}_{s, \pi^N_\cdot}\!\right)^{\!*} G_s\right\rangle + R(N,H, G,s),
    \end{align}
where
    \begin{align}
        R(N,H,G,s) 
        \nn&\leq \frac{1}{N} \sum_{(x,\sigma) \in V_N} \eta(x,\sigma)||\partial_xG||_\infty\sup_{(y,\sigma')\in V_N}\big|H_s(\tfrac{y+\sigma'}{N},\sigma') - H_s(\tfrac{y}{N},\sigma')\big| \\
        \nn&\qquad + \frac{1}{N^2} \sum_{(x,\sigma) \in V_N} \eta(x,\sigma)||\partial_{xx}G||_\infty e^{2||H||_\infty}\\
        &= \frac{1}{N} |\eta^N|\left(||\partial_xG||_\infty \sup_{(y,\sigma')\in V_N}\big|H_s(\tfrac{y+\sigma'}{N},\sigma') - H_s(\tfrac{y}{N},\sigma')\big|+ \frac{1}{N}||\partial_{xx}G||_\infty e^{2||H||_\infty}\right).
    \end{align}
We then conclude that  
    \begin{align}\label{R1}
        \nn&\E^{\varrho,H}_N\left[\sup_{t\in[0,T]} \left|\mathscr{M}_{N,t}^{H,G}(\pi^N_{\cdot}) - \mathscr{M}_t^{H,G}(\pi^N_{\cdot})\right|\right] \\
        \nn&= \E^{\varrho,H}_N\left[ \sup_{t\in[0,T]}\left|\int_0^t R(N,H,G,s)\dd s\right| \right]\\
        & \leq \frac{T}{N} \left(||\partial_xG||_\infty\sup_{(y,\sigma')\in V_N}\big|H_s(\tfrac{y+\sigma'}{N},\sigma') - H_s(\tfrac{y}{N},\sigma')\big| + \frac{1}{N}||\partial_{xx}G||_\infty e^{2||H||_\infty}\right)  \E^{\varrho,H}_N\left[|\eta^N|\right] \to 0,
    \end{align}
where we used that $\E^{\varrho,H}_N\left[|\eta^N|\right] \leq ||\varrho||_\infty N$ and that $H_s$ is continuous on $\mathbb{T}$.
\end{proof}
\begin{Lemma}\label{lemma 2.2}
    For every $G \in C^\infty([0,T]\times V)$ we have that 
    \begin{equation}
        \lim_{N\to\infty} \E^{\varrho,H}_N\left[\sup_{t\in[0,T]} \left(\mathscr{M}_{N,t}^{H,G}(\pi^N_{\cdot})\right)^2\right]=0.
    \end{equation}
\end{Lemma}
\begin{proof}
By Doob's maximal inequality we find that 
    \begin{equation}
        \E^{\varrho,H}_N\left[\sup_{t\in[0,T]} \left(\mathscr{M}_{N,t}^{H,G}(\pi^N_{\cdot})\right)^2\right] \leq 4 \E^{\varrho,H}_N\left[ \left(\mathscr{M}_{N,T}^{H,G}(\pi^N_{\cdot})\right)^2\right] = 4\E^{\varrho,H}_N\left[ \int_0^T \Gamma_{N,t}^{H,G}(\pi^N_{t})\dd t\right],
    \end{equation}
where in the last equality we use that the quadratic variation of the martingale $\mathscr{M}_{N,t}^{H,G}$ is given by the integral of the  carr\'e du champ operator $\Gamma_{H,t}^{N,G}$ defined by
    \begin{equation}\label{cdc}
        \Gamma_{N,t}^{H,G}(\pi^N_{t}) = \mathscr{L}^H_{N,t} (\langle \pi^N_{t}, G_t \rangle )^2 - 2\langle \pi^N_{t}, G_t \rangle \cdot \mathscr{L}^H_{N,t} \langle \pi^N_{t}, G_t \rangle .
    \end{equation}
 For a general jump process generator $Lf(\eta) = \sum_{\eta'} r(\eta,\eta') (f(\eta') - f(\eta))$ we have that 
    \begin{equation}
        Lf^2(\eta) - 2 f(\eta)\cdot Lf(\eta) = \sum_{\eta'} r(\eta,\eta') \big(f(\eta') - f(\eta)\big)^2,
    \end{equation}
and so we find that 
    \begin{align}
        \nn\Gamma_{N,t}^{H,G}(\pi^N_t) &= \frac{1}{N} \sum_{(x,\sigma) \in V_N} \eta^N_t(x,\sigma)e^{H_t(\frac{x+\sigma}{N},\sigma) - H_t(\frac{x}{N},\sigma)} \big(G_t(\tfrac{x+\sigma}{N},\sigma) - G_t(\tfrac{x}{N},\sigma)\big)^2 \\
        &\qquad \qquad \qquad +\frac{1}{N^2} \sum_{(x,\sigma) \in V_N} \eta^N_t(x,\sigma)c(\sigma,m(\pi^N_t)) e^{-\sigma \widetilde{H}_t(x)}\big(G_t(x,-\sigma)- G_t(x,\sigma)\big)^2.
    \end{align}
By the mean value theorem and since $c(\sigma,m)$ is bounded, we can find an upper bound given by
    \begin{align}
        \Gamma_{N,t}^{H,G}(\pi^N_{\cdot}) \leq \mathcal{O}(\tfrac{1}{N^2})\cdot |\eta^N|.
    \end{align}
Using again that $\E^{\varrho,H}_N\left[|\eta^N|\right] \leq ||\varrho||_\infty N$, we then find that 
    \begin{equation}
        \E^{\varrho,H}_N\left[\sup_{t\in[0,T]} \left(\mathscr{M}_{N,t}^{H,G}(\pi^N_{\cdot})\right)^2\right] \leq 4 T \mathcal{O}(\tfrac{1}{N^2})\cdot\E^{\varrho,H}_N\left[|\eta^N|\right] \to 0.
    \end{equation}
\end{proof}
\begin{Proposition}
    $\{\pi^N_{\cdot}:N\in\mathbb{N}\}$ is tight in $D([0,T];\mathcal{M}_V)$.
\end{Proposition}
\begin{proof}
    By Aldous' criteria, we have to show the following:
\begin{enumerate}[label=\textbf{B.\arabic*}]
\item \label{B1}For all $t\in[0,T]$ and $\varepsilon>0$ there exists a compact $K(t, \varepsilon) \subset \mathcal{M}_V$ such that 
    \begin{equation}
        \sup_{N\in \mathbb{N}} \p^{\varrho,H}_N\big(\pi^N_{t}\notin K(t,\varepsilon)\big) \leq \varepsilon.
    \end{equation}
\item \label{B2} For all $\varepsilon>0$
    \begin{equation}
        \lim_{\delta \to 0} \limsup_{N\to\infty} \p^{\varrho,H}_N\big(\omega(\pi^N_{\cdot},\delta)\geq \varepsilon\big) = 0,
    \end{equation}
where
    \begin{equation}
        \omega(\pi^N_{\cdot}, \delta) =\sup\{d(\pi^N_{t},\pi^N_{s})| : s,t \in [0,T], |t-s|<\delta\},
    \end{equation}
with $d$ the metric on $\mathcal{M}_V$ defined for $\alpha,\beta \in \mathcal{M}_V$ as  
    \begin{equation}
        d(\alpha,\beta) = \sum_{j=1}^\infty 2^{-j}\Big(1\wedge \big|\left<\alpha,\phi_j\right> - \left<\beta,\phi_j\right>\big|\Big).
    \end{equation}
\end{enumerate}  
We start with proving \ref{B1}. For every $C>0$, we have that the set 
$K_C = \left\{ \mu \in \mathcal{M}_V: \mu(V)\leq C\right\}$
is compact in $\mathcal{M}_V$. Furthermore,
    \begin{align}
        \p^{\varrho,H}_N(\pi^N_{t} \notin K_C) = \p^{\varrho,H}_N(\pi^N_{t}(V) > C) \leq \frac{1}{C} \E^{\varrho,H}_N[\pi^N_{t}(V)],
    \end{align}
where we used the Markov inequality in the last step. Here 
    \begin{equation}
        \E^{\varrho,H}_N[\pi^N_{t}(V)] = \E^{\varrho,H}_N\left[\frac{1}{N}\sum_{(x,\sigma) \in V_N} \eta_t^N(x,\sigma) \delta_{(\tfrac{x}{N},\sigma)}(V)\right] = \frac{1}{N} \E^{\varrho,H}_N\big[|\eta^N|\big] \leq ||\varrho||_\infty.
    \end{equation}
Therefore 
    \begin{equation}
        \p^{\varrho,H}_N(\pi^N_{t} \notin K_C) \leq \frac{1}{C} ||\varrho||_\infty.
    \end{equation}
Since we took $C$ arbitrarily, we can take $C>||\varrho||_\infty \varepsilon^{-1}$, and \ref{B1} follows. 

To prove \ref{B2}, take $\varepsilon'<\varepsilon$ and  note  that by the Markov inequality we have that 
    \begin{equation}
        \lim_{\delta\to0}\limsup_{N\to\infty} \p^{\varrho,H}_N(\omega(\pi^N_{\cdot},\delta) > \varepsilon) \leq \lim_{\delta\to0}\limsup_{N\to\infty}\frac{1}{\varepsilon} \E^{\varrho,H}_N\left[\omega(\pi^N_{\cdot},\delta)\right] \leq \lim_{\delta\to0}\limsup_{N\to\infty}\frac{1}{\varepsilon'} \E^{\varrho,H}_N\left[\omega(\pi^N_{\cdot},\delta)\right].
    \end{equation}
Now for $\omega(\pi^N_{\cdot},\delta)$ we have that 
    \begin{align}
        \omega(\pi^N_{\cdot},\delta) 
        &= \sup_{\substack{s,t\in[0,T]\\|t-s|<\delta}}\sum_{j=1}^\infty 2^{-j}\Big(1\wedge \big|\langle \pi^N_{t},\phi_j\rangle - \pi^N_{s}(\phi_j)\big|\Big)\nn\\
        &\leq 2^{-m} + \sup_{\substack{s,t\in[0,T]\\|t-s|<\delta}}\sum_{j=1}^m \big|\langle \pi^N_{t},\phi_j\rangle - \pi^N_{s}(\phi_j)\big|,
    \end{align}
where we took $m\in\mathbb{N}$ arbitrarily. Using the martingale $\mathscr{M}_{N,t}^{H,\phi_j}(\pi^N_{\cdot})$, we find that 
    \begin{align}
        \E^{\varrho,H}_N\left[\big|\langle \pi^N_{t},\phi_j\rangle - \pi^N_{s}(\phi_j)\big|\right] 
        &= \E^{\varrho,H}_N\left[\left|\mathscr{M}_{N,t}^{H,\phi_j}(\pi^N_{\cdot}) - \mathscr{M}_{N,s}^{H,\phi_j}(\pi^N_{\cdot}) - \int_s^t \mathscr{L}^H_{N,s}\langle \pi^N_{r},\phi_j\rangle \dd r \right|\right]\nn\\
    &\leq 2\E^{\varrho,H}_N\left[\sup_{t\in[0,T]} \left|\mathscr{M}_{N,t}^{H,\phi_j}(\pi^N_{\cdot})\right|\right] + \E^{\varrho,H}_N\left[\left|\int_s^t \mathscr{L}^H_{N,s}\langle \pi^N_{r},\phi_j\rangle\dd r\right|\right].
    \end{align}
By Lemma \ref{lemma 2.2} the first expectation vanishes as $N\to\infty$. By \eqref{Ln} we can upper bound the second expectation by 
    \begin{equation}
        \E^{\varrho,H}_N\left[\left|\int_s^t \mathscr{L}^H_{N,s}\langle \pi^N_{r},\phi_j\rangle\dd r\right|\right] \leq \E^{\varrho,H}_N\left[\left|\int_s^t\left\langle \pi^N_{r}, \left(\!A^{H}_{s, \pi^N_\cdot}\!\right)^{\!*} \phi_j\right\rangle\dd r\right|\right] + \E^{\varrho,H}_N\left[\left|\int_s^t R(N,\phi_j,r)\dd r\right|\right].
    \end{equation}
From \eqref{R1} we see that the second expectation also vanishes as $N\to\infty$ (uniformly in $s$ and $t$). For the first expectation note that 
    \begin{align}
        \E^{\varrho,H}_N\left[\left|\int_s^t\left\langle \pi^N_{r}, \left(\!A^{H}_{s, \pi^N_\cdot}\!\right)^{\!*} \phi_j\right\rangle\dd r\right|\right] 
        &\leq \E^{\varrho,H}_N\left[\left|\int_s^t \frac{1}{N} \sum_{(x,\sigma)\in V_N}\eta_r^N(x,\sigma) \left|\left|\left(\!A^{H}_{s, \pi^N_\cdot}\!\right)^{\!*}\phi_j\right|\right|_\infty \dd r  \right|\right]\nn\\
        \nn&\leq  \frac{1}{N} \left|\left|\left(\!A^{H}_{s, \pi^N_\cdot}\!\right)^{\!*}\phi_j\right|\right|_\infty |t-s| \cdot \E^{\varrho,H}_N\big[|\eta^N|\big]\\
        &\leq \left|\left|\left(\!A^{H}_{s, \pi^N_\cdot}\!\right)^{\!*}\phi_j\right|\right|_\infty ||\varrho||_\infty \delta .
    \end{align}
Combining  all of the above, we find that 
    \begin{align}
        \lim_{\delta\to0}\limsup_{N\to\infty} \p^{\varrho,H}_N(\omega(\pi^N_{\cdot},\delta)>\varepsilon)
        &\leq  \frac{1}{\varepsilon'} 2^{-m} + \lim_{\delta\to0}\limsup_{N\to\infty} \frac{1}{\varepsilon'} \E^{\varrho,H}_N\left[\sup_{\substack{s,t\in[0,T]\\|t-s|<\delta}}\sum_{j=1}^m \big|\langle \pi^N_{t},\phi_j\rangle - \pi^N_{s}(\phi_j)\big|\right]\nn\\
        \nn&\leq \frac{1}{\varepsilon'} 2^{-m} +\lim_{\delta \to 0} \frac{1}{\varepsilon'}  \left|\left|\left(\!A^{H}_{s, \pi^N_\cdot}\!\right)^{\!*}\phi_j\right|\right|_\infty ||\varrho||_\infty \delta \\
        &= \frac{1}{\varepsilon'} 2^{-m}.
    \end{align}
Since we took $m$ arbitrarily, we can choose it such that $2^{-m} \leq (\varepsilon')^2$, i.e., 
    \begin{equation}
        \lim_{\delta\to0}\limsup_{N\to\infty} \p^{\varrho,H}_N(\omega(\pi^N_{\cdot},\delta)>\varepsilon) < \varepsilon',
    \end{equation}
and since we took $\varepsilon'$ arbitrarily small, we indeed find that \ref{B2} holds. 
\end{proof}
We are now ready to give the proof of the hydrodynamic limit of the weakly perturbed model. 
\begin{proof}[Proof of Theorem \ref{weak hydro theorem}]
By Prokhorov's theorem, the tightness of the sequence $\{\pi^N_{\cdot}:N\in\mathbb{N}\}$ implies that the sequence is sequentially compact. If we then prove that every convergent subsequence converges to $\delta_\alpha$, then the theorem holds. 
    
Take such a convergent subsequence $\p^{\varrho,H}_{N_k}(\pi^N_{\cdot}\in\cdot) \to P^*$, with $P^*$ a probability measure on $D([0,T];\mathcal{M}_V)$.  For a given $\varepsilon>0$  and $G \in C^\infty(V)$,
define the set 
    \begin{equation}
        \Xi^{H,G}_{\varepsilon} = \left\{ \beta \in D([0,T];\mathcal{M}_V) : \sup_{t\in[0,T]} \left|\mathscr{M}^{H,G}_t(\beta)\right| \leq \epsilon \right\},
    \end{equation}
which is closed in the Skorokhod topology. By Portmanteau's theorem, we now have that 
    \begin{align}
        P^*(\Xi^{H,G}_{\varepsilon}) 
        &\nn\geq \lim_{k\to\infty} \p^{\varrho,H}_{N_k} (\pi^{N_k}_{\cdot} \in \Xi^{H,G}_{\varepsilon})\\
        \nn&= \lim_{k\to\infty} \p^{\varrho,H}_{N_k}\left( \sup_{t\in[0,T]} \left|\mathscr{M}_t^{H,G}(\pi^{N_k}_{\cdot})\right|\leq \varepsilon\right)\\
        &= \lim_{k\to\infty} \p^{\varrho,H}_{N_k}\left( \sup_{t\in[0,T]} \left|\mathscr{M}_{N_k,t}^{H,G}(\pi^{N_k}_{\cdot})\right| \leq \varepsilon\right),
    \end{align}
where we used Lemma \ref{lemma 2.1} for the last step. Now by using Chebyshev's inequality together with Lemma \ref{lemma 2.2}, we find that \\
    \begin{equation}
        \p^{\varrho,H}_{N_k}\left( \sup_{t\in[0,T]} \left|\mathscr{M}_{N_k,t}^{H,G}(\pi^{N_k}_{\cdot})\right| > \varepsilon\right) \leq \frac{1}{\varepsilon^2} \E^{\varrho,H}_{N_k}\left[  \sup_{t\in[0,T]} \left|\mathscr{M}_{N_k,t}^{H,G}(\pi^{N_k}_{\cdot})\right|^2 \right] \to 0,
    \end{equation}
and so indeed 
    \begin{equation}
        P^*(\Xi^{H,G}_{\varepsilon}) 
        \geq \lim_{k\to\infty} \p^{\varrho,H}_{N_k}\left( \sup_{t\in[0,T]} \left|\mathscr{M}_{N_k,t}^{H,G}(\pi^{N_k}_{\cdot})\right| \leq \varepsilon\right) = 1.
    \end{equation}
Since this is true for all $\varepsilon>0$ and $G\in C^{\infty}(V)$, it follows that $P^*=\delta_{\alpha^H}$.
\end{proof}\ \\

\end{document}